\documentclass[12pt]{article}
\usepackage{fancyhdr,graphicx}
\usepackage{amsfonts}
\usepackage{amssymb}
\usepackage{amsthm}
\usepackage{newlfont}
\usepackage{mathrsfs}
\usepackage{bbm}

\usepackage{subfigure}
\usepackage{fancybox}

\newtheorem{theorem}{Theorem}[section]

\newtheorem{lemma}[theorem]{Lemma}
\newtheorem{corollary}[theorem]{Corollary}
\newtheorem{proposition}[theorem]{Proposition}
\newtheorem{definition}[theorem]{Definition}

\usepackage{latexsym,bm}
\title{{Matching preclusion for $n$-grid graphs}\thanks{This work was supported by NSFC (Grant No. 11371180).}}
\author{Qi Ding, Heping Zhang
\\ \footnotesize{School of Mathematics and Statistics, Lanzhou University, Lanzhou, Gansu 730000, P. R. China}
\\ \footnotesize{E-mail addresses: dingqi0@yonyou.com, zhanghp@lzu.edu.cn}
\\ Hui Zhou
\\ \footnotesize{School of Mathematical Sciences, Peking University, Beijing 100871, P. R. China}
\\ \footnotesize{E-mail address: zhouhpku13@pku.edu.cn}
}

\begin{document}

\maketitle

\begin{abstract}
A matching preclusion set of a graph is an edge set whose deletion results in a graph without perfect matching or almost perfect matching. The Cartesian product of $n$ paths is called an $n$-grid graph. In this paper, we study the matching preclusion problems for  $n$-grid graphs and obtain the following results. If an $n$-grid graph has an even order, then it has the matching preclusion number  $n$, and every optimal matching preclusion set is trivial. If the $n$-grid graph has an odd order, then it has the matching preclusion number $n+1$, and all the optimal matching preclusion sets are characterized.
\end{abstract}

\textbf{Key words}: Matching preclusion number; Optimal matching preclusion set; $n$-grid graph.



\section{Introduction}\label{section Introduction}

Let $G$ be a simple graph (without multiple edges or loops). We use $V(G)$ and $E(G)$ to denote its vertex set and edge set, respectively. The cardinality of $V(G)$ is called the \emph{order} of graph $G$, and denoted by $n(G)$. An edge set $M\subseteq E(G)$ is called \emph{matching} if no two edges in $M$ have a common end vertex. A matching of $G$ with the maximum cardinality is called a \emph{maximum matching}. For a matching $M$ of $G$, if the edges in $M$ are incident with all vertices of $G$, i.e. $|V(G)|=2|M|$, then $M$ is called a \emph{perfect matching} of $G$. If the edges in $M$ are incident with all but one of the vertices of $G$, i.e. $|V(G)|-1=2|M|$, then $M$ is called an \emph{almost perfect matching} of $G$. A \emph{matching preclusion set} of a graph $G$ is an edge set $F\subseteq E(G)$ such that $G-F$ has neither perfect matching nor almost perfect matching. A matching preclusion set with minimum cardinality is called an \emph{optimal matching preclusion set}. The cardinality of an optimal matching preclusion set of $G$ is called the \emph{matching preclusion number} of $G$ and is denoted by $mp(G)$. Brigham et al.~\cite{3} introduced the matching preclusion number as a measure
of robustness in the event of edge failure in interconnection networks.

The \emph{degree} $d_G(v)$ of a vertex $v$ in $G$ is the number of vertices adjacent to $v$. Let $\delta(G)$ and $\Delta(G)$ denote the \emph{minimum degree} and \emph{maximum degree} of $G$ respectively. Cheng et al. gave the following result.
\begin{proposition}[\cite{8}]\label{pro even order mp number upper bound}
Let $G$ be a graph with even order. Then $mp(G)\leqslant \delta(G)$.
\end{proposition}

To investigate the structure of optimal matching preclusion sets, Cheng et al.~\cite{8} gave the following concepts. A matching preclusion set $F$ of a graph $G$ is \emph{trivial} if all the edges of $F$ are those edges incident with a single vertex or two vertices according as $G$ has even or odd order. If $G$ has an optimal matching preclusion set $F$ that is trivial, then $G$ is called \emph{maximally matched}. If every optimal matching preclusion set of $G$ is trivial, then $G$ is called \emph{super matched}.

The \emph{Cartesian product} $G\Box H$ of graphs $G$ and $H$ is a graph such that the vertex set of $G\Box H$ is $V(G)\times V(H)$, and  two vertices $(u,u')$ and $(v,v')$ are adjacent in $G\Box H$ if and only if either $u=v$ and $u'$ is adjacent to $v'$ in $H$, or $u$ is adjacent to $v$ in $G$ and $u'=v'$.

The matching preclusion numbers and optimal matching preclusion sets of the following graphs were studied: Petersen graph, complete graph $K_n$, complete bipartite graph $K_{n,n}$ and  hypercube $Q_n$ by Brigham et al.~\cite{3}, complete bipartite graph $K_{n,n+1}$ by Cheng et al.~\cite{8}, even order $k$-ary $n$-cube $Q_n^k$ by Wang et al.~\cite{24}, tori and related Cartesian products by Cheng et al.~\cite{14}, balanced hypercube $BH_n$ by L\"{u} et al.~\cite{22}, crossed cube $CQ_n$ by Cheng et al.~\cite{20}.

Cheng et al.~\cite{5,6,24} considered bipartite graphs and showed that the matching preclusion number of a $k$-regular bipartite graph is $k$. The matching preclusion for Cayley graphs can be found in \cite{9,12,16}. Li et al.~\cite{LHZ} considered a general $k$-regular connected vertex-transitive graph with even order and showed that it has matching preclusion number $k$ and is super matched except for six classes of graphs. Almeida et al.~\cite{1} obtained some properties of the matching preclusion number of the Cartesian product graphs and computed the matching preclusion numbers of several interconnection networks. For other references on this topic, see~ \cite{2,7,10,15,17,18,19,21}.

The Cartesian product of $n$ paths is called the \emph{$n$-grid graph}. More precisely, let $n\geq 1$ be an integer. For each $0\leqslant i\leqslant n-1$, let $k_i \geq 2$ be an integer. Let $P_{k_{0}},P_{k_{1}},\ldots,P_{k_{n-1}}$ be paths. Here $P_{k}$ denotes a path with $k$ vertices. Then the Cartesian product $G=P_{k_{0}}\Box P_{k_{1}}\Box\cdots\Box P_{k_{n-1}}$ is called a $(k_0,k_1,\ldots,k_{n-1};n)$-grid graph, or simply an \emph{$n$-grid graph}, where each vertex $u$ of $G$ can be regarded as an $n$-tuple $(u_0,u_1,\ldots,u_{n-1})$ where $0\leqslant u_i\leqslant k_i-1$ for each $0\leqslant i\leqslant n-1$.

In this paper, we discuss the matching preclusion for $n$-grid graphs. The $n$-grid graphs are a generalization of hypercube $Q_n$, but not regular in other cases. In Section~\ref{section $n$-Grid graph}, we will give some structural properties of $n$-grid graphs. A crucial definition, that is $(f;4)$-cycle, is also given. In Section 3, the matching preclusion for even order $n$-grid graph is got. We treat first $2$-grid graphs and then $n$-grid graphs for $n\geqslant 3$. By using these results for even order $n$-grid graphs, we obtain matching preclusion for odd order $n$-grid graphs in Section 4.

The following is our main theorem, and implied by Theorems~\ref{thm even order mp}, \ref{thm odd PM number} and \ref{thm odd PM set}.

\begin{theorem}[Main theorem]
Let $G$ be an $n$-grid graph. If $n(G)$ is even, then $mp(G)=n$, and $G$ is super matched whenever $n\geqslant 3$. If $n(G)$ is odd, then $mp(G)=n+1$, moreover if $n\geqslant 2$, then $F\subseteq E(G)$ is an optimal matching preclusion set of $G$ if and only if $F$ consists of edges incident with a vertex $u=(u_0,u_1,\ldots,u_{n-1})$ with $d_G(u)=n+1$ and $\sum\limits_{i=0}^{n-1}u_i$ odd.
\end{theorem}

The conditional matching preclusion number and optimal conditional matching preclusion sets of $n$-grid graphs will be given elsewhere.

\section{Preliminaries for $n$-grid graphs}\label{section $n$-Grid graph}


\subsection{Basic notations}\label{subsection Basic notations}

Let $G$ be a graph. Let $F,U,H$ be a subset of $E(G)$, a subset of $V(G)$ and a subgraph of $G$, respectively. The subgraph $G-F$ is obtained from $G$ by deleting all the edges in $F$. When $F=\{f\}$ for some edge $f\in E(G)$, we use $G-f$ to denote $G-\{f\}$. The subgraph $G-U$ is obtained from $G$ by deleting all the vertices in $U$ together with their incident edges. When $U=\{u\}$ for some vertex $u\in V(G)$, we use $G-u$ instead of $G-\{u\}$. The subgraph $G-H$ is $G-V(H)$.

Let $f\in E(G)$. We use $d_F(f)$ to denote the size of the set $\{g\in F\mid g$ is incident with $f\}$.

Let $v\in V(G)$. If there exists an edge $e\in F$ such that $e$ is incident with $v$, then we say \emph{$F$ covers $v$}, otherwise we say \emph{$F$ uncovers $v$}. We say \emph{$F$ uncovers $U$} if $F$ uncovers $u$ for any vertex $u\in U$. We say \emph{$F$ uncovers $H$} if $F$ uncovers $V(H)$.

Let $A,B$ be subsets of $X$. The set $A\setminus B=\{x\in X\mid x\in A,x\notin B\}$. The symmetric difference $A\Delta B$ of $A$ and $B$ is $(A\setminus B)\cup (B\setminus A)$. Let $M$ be a matching of $G$. A cycle $C$ is called an \emph{$M$-alternating cycle} if the edges of $C$ alternate in $M$ and $E(G)\setminus M$. We identify the edge set $E(C)$ with the cycle $C$ without confusion and we use $M\Delta C$ to denote the symmetric difference $M\Delta E(C)$.


\subsection{Structure of $n$-grid graphs}\label{subsection Structure of the $n$-grid graph}

We suppose $G$ is a $(k_0,k_1,\ldots,k_{n-1};n)$-grid graph in this subsection. For simplicity, for each $0\leqslant i\leqslant n-1$, we use $V(P_{k_{i}})=\{0,1,\ldots,k_{i}-1\}$ and $E(P_{k_{i}})=\{j(j+1)~|~j\in \{0,1,\ldots,k_{i}-2\}\}$ without confusion. The order of $G$ is $n(G)=k_{0}k_{1}\cdots k_{n-1}$. Two vertices $v=(v_{0},v_{1},\ldots,v_{n-1})$ and $u=(u_{0},u_{1},\ldots,u_{n-1})$ of $G$ are adjacent if and only if there exists an integer $i_{0}\in \{0,1,\ldots,n-1\}$ such that $|v_{i_{0}}-u_{i_{0}}|=1$ and $v_{j}=u_{j}$ for every $j\in\{0,1,\ldots,n-1\}\setminus \{i_{0}\}$. We say that the \emph{position} of the edge $f=uv\in E(G)$ is $i_{0}$ or $u$ and $v$ are adjacent at position $i_{0}$.

Obviously, the maximum degree $\Delta(G)=2n-n_{2}$, where $n_{2}$ is the number of $2$'s in $k_{0},k_{1},\ldots,k_{n-1}$, and the minimum degree $\delta(G)=n$. We see that $\Delta(G)=\delta(G)$ if and only if $n_2=n$ which means $G$ is the hypercube $Q_n$. (This is the only case $G$ is regular.) The graph $G$ is not regular when $n_2<n$. If the degree of a vertex $v$ is $n$, then $v=(v_{0},v_{1},\ldots,v_{n-1})$ satisfies that for each $0\leqslant i\leqslant n-1$, $v_{i}=0$ or $k_{i}-1$. Let \begin{center}$V_{\delta}(G)=\{v\in V(G)\mid d_G(v)=\delta(G)\}$.\end{center} Then we have $|V_{\delta}(G)|=2^{n}$. If the degree of a vertex $v$ is $2n-n_{2}$, without loss of generality we may assume that the first $n_{2}$ numbers of $k_{0},k_{1},\ldots,k_{n-1}$ are $2$'s, then $v=(v_{0},v_{1},\ldots,v_{n-1})$ satisfies that for each $0\leqslant i \leqslant n_2-1$, we have $v_{i}=0$ or $1$, and for each $n_{2}\leqslant j \leqslant n-1$, we have $0<v_{j}<k_{j}-1$. Let \begin{center}$V_{\Delta}(G)=\{v\in V(G)\mid d_G(v)=\Delta(G)\}$.\end{center} Then we have $|V_{\Delta}(G)|=\prod\limits_{i=n_2}^{n-1}(k_i-2)$ when $n_2 < n$, or $|V_{\Delta}(G)|=2^n=|V_{\delta}(G)|$ when $n_2 = n$.

An index $d\in \{0,1,\ldots,n-1\}$ is often referred as a position. Now we fix a position $d\in \{0,1,\ldots,n-1\}$, and define
$E_{d}(G)=\{f\in E(G)\mid$ the position of $f$ is $d\}.$ For each $j\in \{0,1,\ldots,k_{d}-1\}$, let $G_{d}[j]$ be the subgraph of $G$ induced by $\{u\in V(G)\mid u=(u_{0},u_{1},\ldots,u_{n-1}), u_d=j\}.$ Then $G$ has a partition \begin{center}$G = G_{d}[0]\cup G_{d}[1]\cup \cdots\cup G_{d}[k_{d}-1]\cup E_{d}(G)$,\end{center} and the connected components of $G-E_{d}(G)$ are $G_{d}[0],G_{d}[1],\ldots, G_{d}[k_{d}-1]$. In this case we say that $G$ is partitioned at position $d$, or $G_{d}[0]\cup G_{d}[1]\cup \cdots\cup G_{d}[k_{d}-1]\cup E_{d}(G)$ is a partition of $G$ at position $d$. It is easy to see that for each $0\leqslant j\leqslant k_{d}-1$, \begin{center}$G_{d}[j]=P_{k_{0}}\Box\cdots\Box P_{k_{d-1}}\Box \{j\}\Box P_{k_{d+1}}\Box\cdots\Box P_{k_{n-1}}$\end{center} is an $(n-1)$-grid graph. Note that the corresponding vertices in $G_{d}[0],G_{d}[1],\ldots, G_{d}[k_{d}-1]$ are joined through paths of length $k_{d}-1$ at position $d$. For each $j\in \{0,1,\ldots,k_{d}-2\}$, let $E_{d}^{j,j+1}(G)$ be the set of edges between $G_{d}[j]$ and $G_{d}[j+1]$ in $G$, that is $E_{d}^{j,j+1}(G)=\{f\in E(G)\mid$ the position of $f=uv$ is $d$, $u=(u_{0},u_{1},\ldots, u_{n-1}), v=(v_{0},v_{1},\ldots, v_{n-1})$ and $\{u_{d}, v_{d}\}=\{j, j+1\}\}\subseteq E_{d}(G).$ Then \begin{center}$E_{d}(G) = E_{d}^{0,1}(G)\cup E_{d}^{1,2}(G)\cup \cdots\cup E_{d}^{k_{d}-2,k_{d}-1}(G)$.\end{center} Now we have a refined partition \begin{center}$G=G_{d}[0]\cup E_{d}^{0,1}(G)\cup G_{d}[1]\cup E_{d}^{1,2}(G)\cup\cdots\cup G_{d}[k_{d}-2]\cup E_{d}^{k_{d}-2,k_{d}-1}(G)\cup G_{d}[k_{d}-1]$\end{center} of the $n$-grid graph $G$ at position $d$.

For each $j\in \{0,1,\ldots,k_{d}-1\}$, and any vertex $v=(v_{0},v_{1},\ldots, v_{n-1})$ in $G_{d}[j]$, we have $v_d=j$. When $j\geqslant 1$, the corresponding vertex to $v$ in $G_{d}[j-1]$ exists, and we denote it by $v^{-}$, i.e. $v^{-}=(v_{0},v_{1},\ldots,v_{d-1},j-1,v_{d+1},\ldots,v_{n-1})$. When $j\leqslant k_{d}-2$, the corresponding vertex to $v$ in $G_{d}[j+1]$ exists, and we denote it by $v^{+}$, i.e. $v^{+}=(v_{0},v_{1},\ldots,v_{d-1},j+1,v_{d+1},\ldots,v_{n-1})$.

The subgraph $G[E_{d}]$ induced by $E_{d}=E_{d}(G)$ is the union of $n(G)/k_{d}$ disjoint paths of length $k_{d}-1$. If $k_{d}$ is even, then $G[E_{d}]$ contains a perfect matching $M_{d}$ of $G$, where \begin{center}$M_{d}=E_{d}^{0, 1}(G)\cup E_{d}^{2, 3}(G)\cup \cdots \cup E_{d}^{k_{d}-2, k_{d}-1}(G)$.\end{center} 
Generally, if the order $n(G)$ is even, then $G$ has a perfect matching. If the order $n(G)$ is odd, then $G$ has no perfect matching, but it has almost perfect matchings. Several lemmas about these results are in Subsection~\ref{subsection Properties of the $n$-grid graph}. So in the following sections we will discuss the matching preclusion for the $n$-grid graph $G$ in two cases: it has a perfect matching or an almost perfect matching.

To end this subsection, we define an $(f;4)$-cycle (see Figure~\ref{fig: fdjfourcycle}), which will be used frequently in the following sections. We assume $n\geqslant 2$. Let $f\in E(G)$ whose position is $d$ where $0\leqslant d\leqslant n-1$. Suppose $f=uv$, $u=(a_0,a_1,\ldots,a_{d-1},j,a_{d+1},\ldots,a_{n-1})$ and $v=(a_0,a_1,\ldots,a_{d-1},j+1,a_{d+1},\ldots,a_{n-1})$ for some $0\leqslant j\leqslant k_{d}-2$. Then $u\in V(G_d[j])$ and $v\in V(G_d[j+1])$. Let $\delta(G_d[j])$ and $\delta(G_d[j+1])$ be the minimum degree of $V(G_d[j])$ and $V(G_d[j+1])$, respectively. Since $d_{G_d[j]}(u)=d_{G_d[j+1]}(v)\geqslant \delta(G_d[j])=\delta(G_d[j+1]) = n-1\geqslant 1$, there exist vertices $u'=(a_0',a_1',\ldots,a_{d-1}',j,a_{d+1}',\ldots,a_{n-1}')$ and $v'=(a_0',a_1',\ldots,a_{d-1}',j+1,a_{d+1}',\ldots,a_{n-1}')$ such that $u'$ is a neighbor of $u$ in $G_d[j]$ and $v'$ is a neighbor of $v$ in $G_d[j+1]$. Then $u'$ is adjacent to $v'$ at position $d$. Let $C$ be the cycle induced by vertices $u,v,v',u'$. The cycle $C$ is of length $4$. Then we call $C$ an \emph{$(f,d,j;4)$-cycle} or \emph{$(f;4)$-cycle} for short. Let $d'$ be the position of $uu'$. Then the position of $vv'$ is $d'$. In this case, we denote the $(f,d,j;4)$-cycle $C$ as an $(f,d,j,d';4)$-cycle. Fix the edge $f$, the number of $(f,d,j;4)$-cycles is $d_{G_d[j]}(u)=d_{G_d[j+1]}(v)$. Let $M$ be a matching of $G$ such that $E_d^{j,j+1}(G)\subseteq M$, and let $C$ be an $(f,d,j;4)$-cycle. Then $C$ is an $M$-alternating cycle and $f\in M\cap E(C)$. So $f\not\in M\Delta C$.

\newsavebox{\fdjfourcycle}
\savebox{\fdjfourcycle}{
\setlength{\unitlength}{1em}
\begin{picture}(12,9)\label{pic fdjfourcycle}
\put(2.5,5){\oval(3,6)}
\put(9.5,5){\oval(3,6)}
\put(2.5,6){\line(1,0){7}}
\put(2.5,4){\line(1,0){7}}
\put(2.5,4){\line(0,1){2}}
\put(9.5,4){\line(0,1){2}}
\put(2.5,4){\circle*{0.3}}
\put(2.5,6){\circle*{0.3}}
\put(9.5,4){\circle*{0.3}}
\put(9.5,6){\circle*{0.3}}
\put(5.5,6.5){$f$}
\put(5.5,8){$M$}
\put(5.5,4.5){$C$}
\put(1.5,6){$u$}
\put(1.5,3.5){$u'$}
\put(10,3.5){$v'$}
\put(10,6){$v$}
\put(1.5,0.5){$G_d[j]$}
\put(7.5,0.5){$G_d[j+1]$}
\put(4,2){$E_d^{j,j+1}(G)$}
\end{picture}}

\begin{figure}[ht]
\centering
\usebox{\fdjfourcycle}
\caption{An $(f,d,j;4)$-cycle $C$.}
\label{fig: fdjfourcycle}
\end{figure}


\subsection{Properties of $n$-grid graphs}\label{subsection Properties of the $n$-grid graph}

\begin{lemma}
Let $G$ be a $(k_0,k_1,\ldots,k_{n-1};n)$-grid graph. If 
the order $n(G)$ is even, then $G$ has a perfect matching.
\end{lemma}

\begin{proof}
We may suppose $k_{d}$ is even for some $0\leqslant d \leqslant n-1$. Let $M_{d}$ be defined as in the penultimate paragraph of Subsection~\ref{subsection Structure of the $n$-grid graph}. Then $M_{d}$ is a perfect matching of $G$.
\end{proof}

Let $G$ be a $(k_0,k_1,\ldots,k_{n-1};n)$-grid graph and let \begin{center}$V_{allEven}(G)=\biggl\{ u=(u_{0},u_{1},\ldots,u_{n-1})\in V(G)~\biggl|~$ $u_{i}$ is even for each $0\leqslant i \leqslant n-1$ $\biggr\}$.\end{center} Let $d\in \{0,1,\ldots,n-1\}$ be a position and let $u=(u_{0},u_{1},\ldots,u_{n-1})\in V(G)$ such that $u_i$ is even for $i\in\{0,1,\ldots,n-1\}\setminus \{d\}$. Then $G_d[u_d]$ is an $(n-1)$-grid graph, $u\in G_d[u_d]$ and $u\in V_{allEven}(G_d[u_d])$.

The following lemmas are properties of odd order $n$-grid graphs. Let the order $n(G)$ be odd. Then $k_i\geqslant 3$ is odd for $0\leqslant i \leqslant n-1$. So $V_\delta(G)\subseteq V_{allEven}(G)$.

\begin{lemma}\label{lem Alleven apm}
Let $G$ be an odd order $(k_0,k_1,\ldots,k_{n-1};n)$-grid graph, and let $u\in V_{allEven}(G)$. Then $G$ has an almost perfect matching $M_u$ which uncovers $u$.
\end{lemma}

\begin{proof}
Suppose $u=(u_{0},u_{1},\ldots,u_{n-1})\in V_{allEven}(G)$. We will prove this result by induction on $n$. When $n=1$, $G$ is a path $P_{k}=v_{0}v_{1}\cdots v_{k-1}$ of length $k-1$, where $k=k_{0}\geq 3$ is odd and $v_j=j$ for each $0\leqslant j\leqslant k-1$. If $u=v_j$ is even, then $M_u=\{v_{0}v_{1},\ldots,v_{j-2}v_{j-1},v_{j+1}v_{j+2},\ldots,v_{k-2}v_{k-1}\}$ is an almost perfect matching of $G$ such that $M_u$ uncovers $u$.

We suppose that $n\geqslant 2$. Let $d\in \{0,1,\ldots,n-1\}$. Then $G=G_{d}[0]\cup G_{d}[1]\cup\cdots\cup G_{d}[k_{d}-1]\cup E_{d}(G)$. For each $j\in\{0,1,\ldots,k_{d}-1\}$, $G_{d}[j]$ is isomorphic with an odd order $(n-1)$-grid graph. So the result is true for $G_{d}[j]$ by induction hypothesis. Since $u\in V_{allEven}(G_{d}[u_{d}])$, we can find an almost perfect matching $M_d[u_{d}]$ of $ G_{d}[u_{d}]$ such that $M_d[u_{d}]$ uncovers $u$. Then $M_u=E_{d}^{0,1}(G)\cup E_{d}^{2,3}(G)\cup\cdots\cup E_{d}^{u_{d}-2,u_{d}-1}(G)\cup M_d[u_{d}]\cup E_{d}^{u_{d}+1,u_{d}+2}(G)\cup\cdots\cup E_{d}^{k_{d}-2,k_{d}-1}(G)$ is an almost perfect matching of $G$ such that $M_u$ uncovers $u$.
\end{proof}

Let $u\in V_{allEven}(G)$. Then $G-u$ has a perfect matching by Lemma~\ref{lem Alleven apm}. Here we calculate the matching preclusion number of $G-u$.

\begin{lemma}\label{lem mp(G-u) for u in VallEvenG odd order}
Let $G$ be an odd order $(k_0,k_1,\ldots,k_{n-1};n)$-grid graph, and let $u\in V_{allEven}(G)$. Then $mp(G-u)=n$.
\end{lemma}
\begin{proof}
Since $u\in V_{allEven}(G)$, then $\delta (G-u)=n$, and $mp(G-u)\leqslant n$ by Lemma~\ref{pro even order mp number upper bound}. Next we will show, by induction on $n$, that $mp(G-u)\geqslant n$, i.e. for any edge set $F\subseteq E(G-u)$ with $|F|\leqslant n-1$, $(G-u)-F$ has a perfect matching.

Suppose $F\subseteq E(G-u)$ with $|F|\leqslant n-1$. When $n=1$, the set $F=\emptyset$ and $G-u$ has a perfect matching by Lemma~\ref{lem Alleven apm}. Now assume $n\geqslant 2$. Let $f\in F$ with position $d$ for some $d\in \{0,1,\ldots,n-1\}$. Let $u=(u_{0},u_{1},\ldots,u_{n-1})$. Then $G=G_{d}[0]\cup G_{d}[1]\cup \cdots\cup G_{d}[k_{d}-1]\cup E_{d}(G)$, $f\in E_{d}(G)$ and $u\in V_{allEven}(G_{d}[u_d])$. Let $x_{s}$ be an even number for each $s\in \{0,1,\ldots,n-1\}\setminus \{d\}$. For each $j\in \{0,1,\ldots,k_{d}-1\}$, let $G_{j}=G_{d}[j]$, $F_{j}=F\cap E(G_{j})\subseteq F\setminus \{f\}$ and $v_{j}=(x_0,x_1\ldots,x_{d-1}, j,x_{d+1},\ldots, x_{n-1})\in V(G_{j})$. Then $P=v_{0}v_{1}\ldots v_{k_{d}-1}$ is a path of length an even number $k_d-1$, $G_{j}$ is an odd order $(n-1)$-grid graph, $|F_{j}|\leqslant n-2$ and $v_{j}\in V_{allEven}(G_{j})$. For each $s\in \{0,1,\ldots,n-1\}\setminus \{d\}$, $k_s\geqslant 3$ is odd and $0\leqslant x_s\leqslant k_s-1$ is even. Then $x_{s}$ has at least two choices. So there are at least $2^{n-1} > n-1\geqslant |F|$ choices of the path $P=v_{0}v_{1}\ldots v_{k_{d}-1}$, and we can choose one which is disjoint with $F$. Let $P=v_{0}v_{1}\ldots v_{k_{d}-1}$ be such a path that $E(P)\cap F=\emptyset$. Since $u_d$ is even, we have $v_{u_{d}}\in V_{allEven}(P)$. Then $P-v_{u_{d}}$ has a perfect matching $M_{P}$ by Lemma~\ref{lem Alleven apm}. By induction hypothesis, $(G_{u_{d}}-u)-F_{u_{d}}$ has a perfect matching $M_{u,u_{d}}$, and $(G_{j}-v_{j})-F_{j}$ has a perfect matching $M_{u}^{j}$ for each $j\in \{0,1,\ldots,k_{d}-1\}\setminus \{u_{d}\}$. Then $M = M_{u}^{0}\cup M_{u}^{1}\cup \cdots \cup M_{u}^{u_{d}-1}\cup M_{u,u_{d}} \cup M_{u}^{u_{d}+1} \cup \cdots \cup M_{u}^{k_{d}-1}\cup M_{P}$ is a perfect matching of $(G-u)-F$.
\end{proof}

Let
\begin{center}
$V_{e}(G)=\biggl\{u=(u_{0},u_{1},\ldots,u_{n-1})\in V(G)~\biggl|~\sum\limits_{i=0}^{n-1}u_{i} $ is even $\biggr\}$
\end{center}
and
\begin{center}
$V_{o}(G)=\biggl\{u=(u_{0},u_{1},\ldots,u_{n-1})\in V(G)~\biggl|~\sum\limits_{i=0}^{n-1}u_{i} $ is odd $\biggr\}$.
\end{center}
The following two lemmas show that $G$ has an almost perfect matching $M_u$ which uncovers the vertex $u\in V_{e}(G)$, and a nice partition of $V(G)$ into $V_{e}(G)$ and $V_{o}(G)$.

\begin{lemma}\label{lem Sumeven apm}
Let $G$ be an odd order $(k_0,k_1,\ldots,k_{n-1};n)$-grid graph, and let $u\in V_{e}(G)$. Then $G$ has an almost perfect matching $M_u$ such that $M_u$ uncovers $u$.
\end{lemma}

\begin{proof}
Suppose $u = (u_{0},u_{1},\ldots,u_{n-1})\in V_e(G)$ and let $n_{o}$ be the number of odd numbers in $u_{0}$, $u_{1}$, $\ldots$, $u_{n-1}$. Then $n_{o}\leqslant n$. We will prove this lemma by induction on $n_{o}$. As $\sum\limits_{i=0}^{n-1}u_{i}$ is even, so $n_{o}$ is even. The base case that the number of odd numbers is zero, holds by Lemma~\ref{lem Alleven apm}. So we suppose that $n_{o}\geqslant 2$. Without loss of generality, we may assume that the numbers $u_{0},u_{1},\ldots,u_{n_{o}-1}$ are odd and the numbers $u_{n_{o}},\ldots,u_{n-1}$ are even.

Now $G$ has a partition at position $0$: $G=G_{0}[0]\cup G_{0}[1]\cup\cdots\cup G_{0}[u_{0}-1]\cup G_{0}[u_{0}]\cup G_{0}[u_{0}+1]\cup\cdots\cup G_{0}[k_{0}-1]\cup E_{0}(G)$. Then $u\in V(G_{0}[u_{0}])$ and $M_{0}=E_{0}^{0,1}(G)\cup E_{0}^{2,3}(G)\cup\cdots\cup E_{0}^{u_{0}-3,u_{0}-2}(G)\cup E_{0}^{u_{0}+2,u_{0}+3}(G)\cup\cdots\cup E_{0}^{k_{0}-2,k_{0}-1}(G)$ is a matching of $G$ such that $M_0$ uncovers subgraphs $G_{0}[u_{0}-1]$, $G_{0}[u_{0}]$ and $G_{0}[u_{0}+1]$ of $G$. These three subgraphs are shown in Figure~\ref{fig: Partition at position 0}.

\newsavebox{\partitionatpositionzero}
\savebox{\partitionatpositionzero}{
\setlength{\unitlength}{1em}
\begin{picture}(15,10)\label{pic partitionatpositionzero}

\put(2.5,5){\oval(3,6)}
\put(7.5,5){\oval(3,6)}
\put(12.5,5){\oval(3,6)}

\put(2.5,3){\circle*{0.3}}
\put(7.5,3){\circle*{0.3}}
\put(7.5,5){\circle*{0.3}}
\put(7.5,7){\circle*{0.3}}
\put(12.5,7){\circle*{0.3}}

\put(2.5,3){\line(1,0){5}}
\put(7.5,7){\line(1,0){5}}

\put(2,3.5){$v^-$}
\put(8,2.75){$v$}
\put(6.5,4.75){$u$}
\put(6.5,6.75){$w$}
\put(12,6){$w^+$}

\put(0,0.5){$G_0[u_0-1]$}
\put(5,0.5){$H=G_0[u_0]$}
\put(10.5,0.5){$G_0[u_0+1]$}

\put(1.5,8.5){$M_0^-$}
\put(6.5,8.5){$M_0^S$}
\put(11.5,8.5){$M_0^+$}
\end{picture}}

\newsavebox{\partitionatpositionone}
\savebox{\partitionatpositionone}{
\setlength{\unitlength}{1em}
\begin{picture}(15,10)\label{pic partitionatpositionone}

\put(2.5,5){\oval(3,6)}
\put(7.5,5){\oval(3,6)}
\put(12.5,5){\oval(3,6)}

\put(2.5,7){\circle*{0.3}}
\put(7.5,5){\circle*{0.3}}
\put(7.5,7){\circle*{0.3}}
\put(12.5,7){\circle*{0.3}}

\put(2.5,7){\line(1,0){5}}
\put(7.5,7){\line(1,0){5}}

\put(2,6){$v$}
\put(7,4){$u$}
\put(12,6){$w$}

\put(0,0.5){$H_1[u_1-1]$}
\put(6.25,0.5){$H_1[u_1]$}
\put(10.5,0.5){$H_1[u_1+1]$}

\put(1.5,8.5){$M_1^-$}
\put(6.5,8.5){$M_{01}$}
\put(11.5,8.5){$M_1^+$}
\end{picture}}

\begin{figure}[ht]
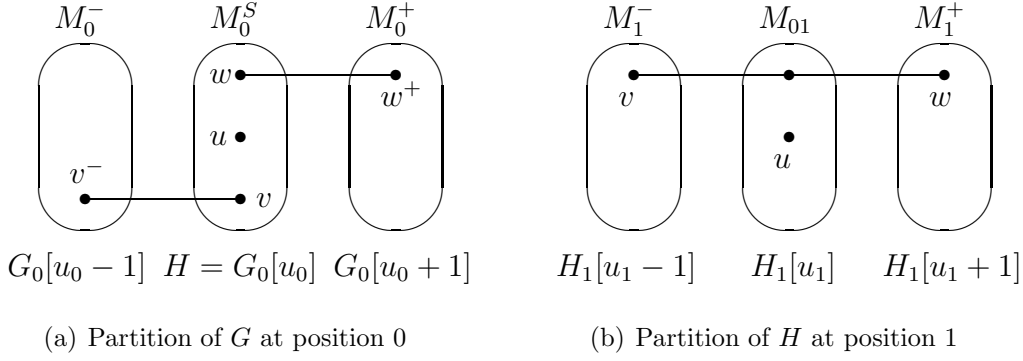

\centering
\subfigure[Partition of $G$ at position $0$]{
\label{fig: Partition at position 0}
\usebox{\partitionatpositionzero}}
\hspace{1em}
\subfigure[Partition of $H$ at position $1$]{
\label{fig: Partition at position 1}
\usebox{\partitionatpositionone}}
\caption{Constructions of  matching $M_0^S$ and  almost perfect matching $M_u$.}
\label{fig: Constructions of the matching M_0^S and the almost perfect matching M_u}
\end{figure}

Let $H=G_{0}[u_{0}]$. So $H=\{u_0\}\Box P_{k_1}\Box P_{k_2}\Box\cdots\Box P_{k_{n-1}}$ is an $(n-1)$-grid graph and $H$ has a partition at position $1$: $H=H_1[0]\cup H_1[1]\cup\cdots \cup H_1[u_{1}-1]\cup H_1[u_{1}]\cup H_1[u_{1}+1]\cup\cdots\cup H_1[k_{1}-1]\cup E_{1}(H)$. Then $u\in V(H_1[u_{1}])$ and $M_{1}=E_{1}^{0,1}(H)\cup E_{1}^{2,3}(H)\cup\cdots\cup E_{1}^{u_{1}-3,u_{1}-2}(H)\cup E_{1}^{u_{1}+2,u_{1}+3}(H)\cup \cdots\cup E_{1}^{k_{1}-2,k_{1}-1}(H)$ is a matching of $H$ such that $M_1$ uncovers subgraphs $H_1[u_{1}-1]$, $H_1[u_{1}]$ and $H_1[u_{1}+1]$ of $H$. These three subgraphs are shown in Figure~\ref{fig: Partition at position 1}.

Now $H_1[u_{1}]=\{u_0\}\Box \{u_1\}\Box P_{k_2}\Box\cdots\Box P_{k_{n-1}}$ is an $(n-2)$-grid graph. The number of odd numbers in $u_{2},\ldots$, $u_{n-1}$ is $n_o-2$ and $\sum\limits_{i=2}^{n-1}u_i$ is even. So $u\in V_e(H_1[u_{1}])$. By the induction hypothesis, we can find an almost perfect matching $M_{01}$ of $H_1[u_{1}]$ such that $M_{01}$ uncovers $u$. Let $v=(u_{0},u_{1}-1,v_{2},\ldots,v_{n-1})\in V(H)$ and let $w=(u_{0},u_{1}+1,v_{2},\ldots,v_{n-1})\in V(H)$ such that $v_{i}$ is even for each $i\in\{2,\ldots,n-1\}$. Then $v\in V_{allEven}\biggl(H_1[u_1-1]\biggr)$ and $w\in V_{allEven}\biggl(H_1[u_1+1]\biggr)$. Then by Lemma~\ref{lem Alleven apm}, we can find an almost perfect matching $M_{1}^{-}$ of $H_1[u_{1}-1]$ and an almost perfect matching $M_{1}^{+}$ of $H_1[u_{1}+1]$ such that $M_1^-$ uncovers $v$ and $M_1^+$ uncovers $w$. Let $S=\{u,v,w\}$. Then $M_0^S=M_1\cup M_1^-\cup M_{01}\cup M_1^+$ is a matching of $H=G_{0}[u_{0}]$ such that $M_0^S$ uncovers $S$.

We go back to the partition of $G$ at position $0$. Let $v^{-}=(u_{0}-1,u_{1}-1,v_{2},\ldots,v_{n-1})\in V(G)$ be the corresponding vertex of $v\in V(G_0[u_0])$ in $G_0[u_0-1]$ and let $w^{+}=(u_{0}+1,u_{1}+1,v_{2},\ldots,v_{n-1})\in V(G)$ be the corresponding vertex of $w\in V(G_0[u_0])$ in $G_0[u_0+1]$. Then $v^{-}\in V_{allEven}\biggl(G_0[u_0-1]\biggr)$ and $w^{+}\in V_{allEven}\biggl(G_0[u_0+1]\biggr)$. By Lemma~\ref{lem Alleven apm}, we can find an almost perfect matching $M_{0}^{-}$ of $G_{0}[u_{0}-1]$ and an almost perfect matching $M_{0}^{+}$ of $G_{0}[u_{0}+1]$ such that $M_0^-$ uncovers $v^{-}$ and $M_0^+$ uncovers $w^{+}$.

Hence $M_u=M_{0}\cup M_{0}^{-}\cup M_0^S\cup M_{0}^{+}\cup \{vv^{-},ww^{+}\}$ is an almost perfect matching of $G$ which uncovers $u$.
\end{proof}

\begin{lemma}\label{lem Bipartite partition}
Let $G$ be an odd order $(k_0,k_1,\ldots,k_{n-1};n)$-grid graph. Then $G$ has a bipartite partition $V(G)=V_{e}(G)\cup V_{o}(G)$, $V_{e}(G)\cap V_{o}(G)=\emptyset$ and $|V_{e}(G)|=|V_{o}(G)|+1$. If $u\in V_{o}(G)$, then $G-u$ has no perfect matching.
\end{lemma}

\begin{proof}
Let $V_{e}=V_{e}(G)$ and $V_{o}=V_{o}(G)$. For any vertex $v=(v_{0},v_{1},\ldots,v_{n-1})\in V(G)$, the number $\sum\limits_{i=0}^{n-1}v_{i}$ is either even or odd, but it cannot be both even and odd at the same time. Hence $V_{e}\cup V_{o}= V(G)$ and $V_{e}\cap V_{o}=\emptyset$. Let $u,v\in V_{e}$(or $V_{o}$), then $uv\notin E(G)$ obviously. Then $G$ is a bipartite graph with bipartition $(V_e,V_o)$.

Assume $u\in V_{e}$. By Lemma~\ref{lem Sumeven apm}, we can find an almost perfect matching $M_{u}$ which uncovers $u$. For any edge $f\in M_{u}$, one of the endpoints of $f$ must belong to $V_{e}$ and the other belongs to $V_{o}$. As $M_{u}$ is an almost perfect matching and it uncovers a vertex in $V_e$, then we have $|V_{e}|=|M_{u}|+1$ and $|V_{o}|=|M_{u}|$. Hence $|V_{e}|=|V_{o}|+1$.

Suppose $u\in V_{o}(G)$ and $G-u$ has a perfect matching $M$. Then $|V_{e}|=|M|$ and $|V_{o}|=|M|+1$. So $|V_{e}|+1=|V_{o}|$. This is a contradiction.
\end{proof}


\section{Matching preclusion for even order $n$-grid graphs}\label{section Matching preclusion of even order $n$-grid graphs}

We first give the matching preclusion number of even order $n$-grid graphs.

\begin{lemma}\label{lem mp number even order}
Let $G$ be an even order $(k_0,k_1,\ldots,k_{n-1};n)$-grid graph. Then $mp(G)=n$.
\end{lemma}

\begin{proof}
Since $\delta (G)=n$, we have $mp(G)\leqslant n$ by Lemma~\ref{pro even order mp number upper bound}. By induction on $n$, we will show that $mp(G)\geqslant n$, i.e. for any edge set $F\subseteq E(G)$ with $|F|\leqslant n-1$, $G-F$ has a perfect matching. Let $F\subseteq E(G)$ with $|F|\leqslant n-1$. If $n=1$, then $F=\emptyset$ and $G$, a path on even number of vertices, has a perfect matching. Now suppose $n\geqslant 2$. Without loss of generality, we may suppose that $n(G)/k_0$ is even. Then $G_0[j]$ is an even order $(n-1)$-grid graph for $0\leqslant j\leqslant k_0-1$. There are two cases.

First suppose $F\subseteq E(G_0[j])$ for some $0\leqslant j\leqslant k_0-1$. If $k_0$ is even, then $M_0=E_{0}^{0,1}(G)\cup E_{0}^{2,3}(G)\cup \cdots\cup E_{0}^{k_{0}-2,k_{0}-1}(G)$ is a perfect matching of $G-F$. When $k_0$ is odd, we consider two possibilities for $j$. If $j\neq 0$, then $M_1\cup M^0$ is a perfect matching of $G-F$ where $M_1=E_{0}^{1,2}(G)\cup E_{0}^{3,4}(G)\cup \cdots\cup E_{0}^{k_{0}-2,k_{0}-1}(G)$ and $M^0$ is a perfect matching of $G_0[0]$. If $j\neq k_0-1$, then $M_2\cup M^{k_0-1}$ is a perfect matching of $G-F$ where $M_2=E_{0}^{0,1}(G)\cup E_{0}^{2,3}(G)\cup \cdots\cup E_{0}^{k_{0}-3,k_{0}-2}(G)$ and $M^{k_0-1}$ is a perfect matching of $G_0[k_0-1]$.

The second case is that $F\nsubseteq E(G_0[j])$ for any $0\leqslant j\leqslant k_0-1$. For $0\leqslant j\leqslant k_0-1$, let $F_j=F\cap E(G_0[j])$. Then $|F_j|\leqslant n-2$. By induction hypothesis, we have $mp(G_0[j])\geqslant n-1 > |F_j|$. Then $G_0[j]-F_j$ has a perfect matching $M[j]$. Hence $M[0]\cup M[1]\cup \cdots \cup M[k_0-1]$ is a perfect matching of $G-F$.
\end{proof}

In the following, we analyze the structure of optimal matching preclusion sets of even order $n$-grid graphs. Let $G$ be an even order $(k_0,k_1,\ldots,k_{n-1};n)$-grid graph. For $n=1$, $G$ is a path $P_{k}$ of length $k-1$ where $k=k_{0}$ is even, and $mp(G)=1$. Then for any optimal matching preclusion set $F\subseteq E(G)$, $P_{k}-F$ is isomorphic to the disjoint union $P_{s}\cup P_{t}$ of two paths $P_{s}$ and $P_{t}$ for some positive odd numbers $s$ and $t$ with $s+t=k$. The structure of optimal matching preclusion sets of $G$ are in Subsection~\ref{subsection $2$-Grid graphs} if $n=2$, and in Subsection~\ref{subsection $n$-Grid graph for n geqslant 3} if $n\geqslant 3$.

\subsection{Even order $2$-grid graphs}\label{subsection $2$-Grid graphs}

\begin{lemma}\label{lem 2-grid PM set}
Let $G$ be an even order $(k_0,k_1;2)$-grid graph where $k_0\geqslant 2$ is even. Let $F$ be an optimal matching preclusion set of $G$. Then either $F$ is trivial or $k_1=3$ and the endpoints of edges of $F$ are $(u_0,0)$, $(u_0+1,0)$, $(u_0,2)$ and $(u_0+1,2)$ for some even number $u_0\in \{0, 2,\ldots, k_0-2\}$.
\end{lemma}

\begin{proof}
Let $F=\{f_0,f_1\}$ and let $M_0=E_0^{0,1}(G)\cup E_0^{2,3}(G)\cup\cdots\cup E_0^{k_0-2,k_0-1}(G)$. Then $F\cap M_0\neq \emptyset$. As if $F\cap M_0=\emptyset$, then $M_0$ is a perfect matching of $G-F$, which is a contradiction. Then $|F\cap M_0|=1$ or $2$.

We begin with $|F\cap M_0|=1$. We may assume that $f_0\in F\cap M_0$ and $f_1\not\in M_0$, then we first show that $f_0$ and $f_1$ have a common endpoint. If not, then there is one $(f_0;4)$-cycle $C_0$ such that $f_1\notin C_0$. So $C_0$ is an $M_0$-alternating cycle and $M_0\Delta C_0$ is a perfect matching of $G-F$ which is a contradiction. Now let $f_0=u_ju_{j+1}$, $u_j\in V(G_0[j])$ and $u_{j+1}\in V(G_0[j+1])$ where $j$ is an even number such that $0\leqslant j\leqslant k_0-2$. If $f_1\not\in E(G_0[j])$ and $f_1\not\in E(G_0[j+1])$, then $f_1\notin C_0$ and $M_0\Delta C_0$ is a perfect matching of $G-F$ which is a contradiction. So $f_1\in E(G_0[j])$ or $f_1\in E(G_0[j+1])$.

\newsavebox{\eightsixtwogrid}
\savebox{\eightsixtwogrid}{
\setlength{\unitlength}{1em}
\begin{picture}(16,12)\label{pic eightsixtwogrid}

\multiput(1,1)(2,0){8}{\line(0,1){10}}
\multiput(1,1)(0,2){6}{\line(1,0){14}}

\put(0.75,0){$0$}
\put(2.75,0){$1$}
\put(4.75,0){$2$}
\put(6.75,0){$3$}
\put(8.75,0){$4$}
\put(10.75,0){$5$}
\put(12.75,0){$6$}
\put(14.75,0){$7$}

\put(0.25,0.75){$0$}
\put(0.25,2.75){$1$}
\put(0.25,4.75){$2$}
\put(0.25,6.75){$3$}
\put(0.25,8.75){$4$}
\put(0.25,10.75){$5$}

\put(9,1){\circle*{0.3}}
\put(11,1){\circle*{0.3}}
\put(9,3){\circle*{0.3}}

\put(9.05,1.75){$f_1$}
\put(9.5,0.2){$f_0$}

\multiput(3.75,0.75)(0,2){5}{$\lozenge$}
\multiput(7.75,0.75)(0,2){5}{$\lozenge$}
\multiput(11.75,0.75)(0,2){5}{$\lozenge$}

\multiput(0.75,1.75)(0,4){3}{$\square$}
\multiput(14.5,1.75)(0,4){3}{$\square$}
\multiput(3.75,10.75)(4,0){3}{$\square$}
\end{picture}}

\newsavebox{\nontrivialoptimalmpset}
\savebox{\nontrivialoptimalmpset}{
\setlength{\unitlength}{1em}
\begin{picture}(12,12)\label{pic nontrivialoptimalmpset}

\multiput(1,7)(2,0){6}{\line(0,1){4}}
\multiput(1,7)(0,2){3}{\line(1,0){10}}
\put(0.75,6){$0$}
\put(2.75,6){$1$}
\put(4.75,6){$2$}
\put(6.75,6){$3$}
\put(8.75,6){$4$}
\put(10.75,6){$5$}
\put(0.25,6.75){$0$}
\put(0.25,8.75){$1$}
\put(0.25,10.75){$2$}
\put(5,7){\circle*{0.3}}
\put(7,7){\circle*{0.3}}
\put(5,11){\circle*{0.3}}
\put(7,11){\circle*{0.3}}
\put(5.5,10.15){$f_1$}
\put(5.5,7.35){$f_0$}

\multiput(1,1)(2,0){6}{\line(0,1){4}}
\multiput(1,1)(0,2){3}{\line(1,0){4}}
\multiput(7,1)(0,2){3}{\line(1,0){4}}
\put(5,3){\line(1,0){2}}
\put(0.75,0){$0$}
\put(2.75,0){$1$}
\put(4.75,0){$2$}
\put(6.75,0){$3$}
\put(8.75,0){$4$}
\put(10.75,0){$5$}
\put(0.25,0.75){$0$}
\put(0.25,2.75){$1$}
\put(0.25,4.75){$2$}
\put(5,1){\circle*{0.3}}
\put(7,1){\circle*{0.3}}
\put(5,5){\circle*{0.3}}
\put(7,5){\circle*{0.3}}
\end{picture}}

\begin{figure}[ht]
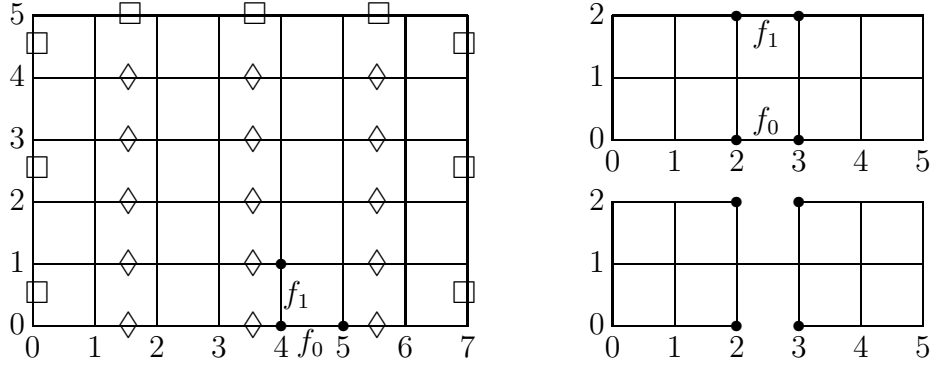

\centering
\subfigure[The $(8,6;2)$-grid graph $G$, and the perfect matching $M_F^0=M_0''\cup M_P$ of $G-F$ when $|F\cap M_0|=1$, where the edges with $\lozenge$ are in $M_0''$ and the edges with $\square$ are in $M_P$.]{
\label{fig: (86)2-grid}
\usebox{\eightsixtwogrid}}
\hspace{1em}
\subfigure[The $(6,3;2)$-grid graph $G$, one nontrivial optimal matching preclusion set $F=\{f_0,f_1\}$ and the resulting graph $G-F$.]{
\label{fig: (63)2-grid}
\usebox{\nontrivialoptimalmpset}}
\caption{Examples of the $(k_0,k_1;2)$-grid graph where $k_0\geqslant 2$ is even.}
\label{fig: Examples of the even order 2-grid graph}
\end{figure}

Suppose $f_1\in E(G_0[j])$. Then $f_1\not\in E(G_0[j+1])$ and $u_j$ is the common endpoint of $f_0$ and $f_1$. If $d_{G_0[j]}(u_j)=2$, then there is exactly one $(f_0;4)$-cycle $C_1$ such that $f_1\notin C_1$. So $M_0\Delta C_1$ is a perfect matching of $G-F$ which is a contradiction. So $d_{G_0[j]}(u_j)=1$ and $u_j$ is an isolated vertex of $G_0[j]-f_1$. Let $u_j=(j,h)$. Then $h=0$ or $k_1-1$. Now we assume $j\neq 0$. Then $k_0\geqslant 4$. Let $M_0'=E_0^{1,2}(G)\cup\cdots \cup E_0^{k_0-3,k_0-2}$. Then $F\cap M_0'=\emptyset$, $M_0\cap M_0'=\emptyset$ and the sets $M_0$ and $M_0'$ form a partition of $E_0(G)$. For each $0\leqslant j_0\leqslant k_0-1$, the subgraph $G_0[j_0]=\{j_0\}\Box P_{k_1}$, and for each $0\leqslant j_1\leqslant k_1-1$, the subgraph $G_1[j_1]=P_{k_0}\Box \{j_1\}$. If $h=0$, then $E_{0,k_1-1}=M_0' \cup E_{010}$ contains a perfect matching $M_F^0$ of $G-F$ where $E_{010}=E(\{0\}\Box P_{k_1})\cup E(P_{k_0}\Box \{k_1-1\})\cup E(\{k_0-1\}\Box P_{k_1})$. Next we construct $M_F^0$. Since $F\cap E_{0,k_1-1}=\emptyset$, any perfect matching contained in $E_{0,k_1-1}$ is a perfect matching of $G-F$. Let $V_{010}=V(\{0\}\Box P_{k_1})\cup V(P_{k_0}\Box \{k_1-1\})\cup V(\{k_0-1\}\Box P_{k_1})$. Then the induced subgraph of $G$ on the set $V_{010}$ is a path $P$ with $k_0+2(k_1-1)$ vertices. Here $k_0+2(k_1-1)$ is even, $V(P)=V_{010}$ and $E(P)=E_{010}$. The path $P$ has a perfect matching $M_P\subset E_{010}$. Let $M_0''=M_0'\setminus E(P_{k_0}\Box \{k_1-1\})$. Then $M_0''\subset M_0'$ is a matching of $G$ such that $M_0''$ uncovers $V_{010}$. Hence $M_F^0=M_0''\cup M_P\subset E_{0,k_1-1}$ is a perfect matching of $G-F$. One example of the construction of $M_F^0$ is in Figure~\ref{fig: (86)2-grid}. If $h=k_1-1$, then by a similar way as in the case $h=0$, we can construct a perfect matching $M_F^{k_1-1}$ of $G-F$ such that $M_F^{k_1-1}\subset E_{0,0}=M_0'\cup E(\{0\}\Box P_{k_1})\cup E(P_{k_0}\Box \{0\})\cup E(\{k_0-1\}\Box P_{k_1})$. Thus $j=0$. The vertex degree of $u_j$ is $d_G(u_j)=2$. Then $u_j$ is an isolated vertex of $G-F$, i.e. the optimal matching preclusion set $F$ is trivial.

By the same argument as in the last paragraph, if $f_1\in E(G_0[j+1])$, then $f_1\not\in E(G_0[j])$, $u_{j+1}$ is the common endpoint of $f_0$ and $f_1$, $j+1=k_0-1$, $d_G(u_{j+1})=2$ and $u_{j+1}$ is an isolated vertex of $G-F$, i.e. the optimal matching preclusion set $F$ is trivial.

Now suppose $|F\cap M_0|=2$. Let $x=(x_0,x_1)$ and $x^+=(x_0+1,x_1)$ be the endpoints of $f_0$, and let $y=(y_0,y_1)$ and $y^+=(y_0+1,y_1)$ be the endpoints of $f_1$. Then $x_0$ and $y_0$ are even. If $x_0\neq y_0$, then there exist a $(f_0;4)$-cycle $C'_0$ and a $(f_1;4)$-cycle $C'_1$ such that $C'_0\cap C_1'=\emptyset$. Thus $M_0\Delta C'_0\Delta C'_1$ is a perfect matching of $G-F$. This is a contradiction. So $x_0=y_0$ and $x_1\neq y_1$. If $k_1-1\notin \{x_1,y_1\}$, then by the same construction, $M_F^0\subset E_{0,k_1-1}$ (defined above) is a perfect matching of $G-F$. If $0\notin \{x_1,y_1\}$, then similarly, $M_F^{k_1-1}\subset E_{0,0}$ (defined above) is a perfect matching of $G-F$. Thus $\{x_1,y_1\}=\{0,k_1-1\}$. If $k_1\geqslant 2$ is even, then $M_e=E_1^{0,1}(G)\cup E_1^{2,3}(G)\cup\cdots\cup E_1^{k_1-2,k_1-1}(G)$ is a perfect matching of $G-F$. So $k_1\geqslant 3$ is odd. If $k_1\geqslant 5$, then $M_o'=E_1^{0,1}(G)\cup E_1^{2,3}(G)\cup \cdots \cup E_1^{k_1-5,k_1-4}(G)\cup E_1^{k_1-2,k_1-1}(G)$ is a matching of $G$ such that $M_o'$ uncovers $G_1[k_1-3]$, and $F\cap E(G_1[k_1-3])=\emptyset$. The subgraph $G_1[k_1-3]$ is a path with even number $k_0$ vertices. So $G_1[k_1-3]$ has a perfect matching $P_o$. Thus $M_o=M_o'\cup P_o$ is a perfect matching of $G-F$, which is a contradiction. So $k_1=3$. Now the endpoints of edges of $F$ are $(x_0,0)$, $(x_0+1,0)$, $(x_0,2)$ and $(x_0+1,2)$, where $0\leqslant x_0\leqslant k_0-2$, and $x_0$ is even. One example of such kind of optimal matching preclusion set is shown in Figure~\ref{fig: (63)2-grid}.
\end{proof}

\begin{lemma}\label{lem special 2-grid PM set}
Let $G$ be an even order $(k_0,3;2)$-grid graph where $k_0\geqslant 2$ is even. Let $F=\{x_0y_0,x_1y_1\}$ where $x_0=(u_0,0)$, $y_0=(u_0+1,0)$, $x_1=(u_0,2)$ and $y_1=(u_0+1,2)$ for some even number $u_0\in \{0, 2,\ldots, k_0-2\}$. Then $F$ is an optimal matching preclusion set of $G$.
\end{lemma}

\begin{proof}
The graph $H=G-F$ has a bridge $f=xy$ where $x=(u_0,1)$ and $y=(u_0+1,1)$. The connected components $H_1$ and $H_2$ of $H-f$ are odd order $2$-grid graphs where $x\in V(H_1)$ and $y\in V(H_2)$. Since $x\in V_{o}(H_1)$, there exists no almost perfect matching of $H_1$ which uncovers $x$ by Lemma~\ref{lem Bipartite partition}. Suppose $M$ is a perfect matching of $H$. Then $f\in M$. So $M_1=M\cap E(H_1)$ is an almost perfect matching of $H_1$ such that it uncovers $x$. This is a contradiction. Thus $H$ has no perfect matching.
\end{proof}

Let $G$ be a $(k_0,k_1;2)$-grid graph where $k_0\geqslant 2$ is even. Let $M_0=E_0^{0,1}(G)\cup E_0^{2,3}(G)\cup\cdots\cup E_0^{k_0-2,k_0-1}(G)$. Then $M_0$ is a perfect matching of $G$. Let $F=\{f_0,f_1\}\subset E(G)$. In the proof of Lemma~\ref{lem 2-grid PM set}, we use  $(f_0;4)$-cycle and  $(f_1;4)$-cycle frequently. When we consider an $(f;4)$-cycle $C$ for some $f\in F\cap M_0$, it follows that $C$ is an $M_0$-alternating cycle, $f\in C$ and for any $f'\in F\setminus \{f\}$, $f'\not\in C$, i.e. $\{f\}=C\cap F\neq \emptyset$, and $C\cap F\subseteq M_0\cap F$. Thus $M_0\Delta C$ is also a perfect matching of $G$, and it contains less edges in $F$ than that of $M_0$, i.e. $F\cap (M_0\Delta C)=(F\cap M_0)\setminus \{f\}$ is a proper subset of $F\cap M_0$. In general, we give the definition of the \emph{$(F,M)$-nice cycle} and some propositions of the nice cycle.

\begin{definition}[The $(F,M)$-nice cycle]\label{definition nice cycle}
Let $G$ be a graph. Let $F\subseteq E(G)$  and  $M$ a matching of $G$ such that $F\cap M\neq \emptyset$. An $M$-alternating cycle $C$ is called an $(F,M)$-nice cycle if $\emptyset\neq C\cap F\subseteq M$.
\end{definition}

\begin{lemma}\label{lem nice cycle}
Let $G$ be a graph, let $F\subseteq E(G)$, and let $M\subseteq E(G)$ be a matching of $G$ such that $F\cap M\neq \emptyset$. Let $C$ be an $(F,M)$-nice cycle. Then $|M\Delta C|=|M|$, $(M\Delta C)\cap F=(M\cap F)\setminus (C\cap F)$ and so $|(M\Delta C)\cap F|=|M\cap F|-|C\cap F|<|M\cap F|$.
\end{lemma}

\begin{proof}
The cycle $C$ is $M$-alternating, then $|M\Delta C|=|M|$. Since $C\cap F\subseteq M\cap F$, we have $(M\Delta C)\cap F\subseteq (M\cup C)\cap F=(M\cap F)\cup (C\cap F)=M\cap F$. Suppose $f\in C\cap F$. Then $f\in C,F,M$. So $f\notin M\Delta C$ and $f\notin (M\Delta C)\cap F$. Hence $(M\Delta C)\cap F\subseteq (M\cap F)\setminus (C\cap F)$. Let $g\in (M\cap F)\setminus (C\cap F)$. Then $g\in M,F$ and $g\notin C$. So $g\in M\Delta C$ and $g\in (M\Delta C)\cap F$. Thus $(M\Delta C)\cap F\subseteq (M\Delta C)\cap F$.
\end{proof}

Let $G$ be a graph and let $F\subseteq E(G)$. Let $f\in F$ be an edge of $F$. Then we call $f$ an \emph{$F$-fault edge}. Let $g\in E(G)\setminus F$ be an edge of $G-F$. Then we call $g$ an \emph{$F$-good edge}. Let $M\subseteq E(G)$ be a matching of $G$ such that $F\cap M\neq \emptyset$. Let $C$ be an $(F,M)$-nice cycle. By Lemma~\ref{lem nice cycle}, the matching $M\Delta C$ of $G$ contains less $F$-fault edges than that of $M$.

Let $n\geqslant 2$ and let $G$ be a $(k_0,k_1,\ldots,k_{n-1};n)$-grid graph. Let $f=uv$ be an edge of $G$ and let $C=uvv'u'u$ be a $(f,d,j;4)$-cycle as defined at the end of Subsection~\ref{subsection Structure of the $n$-grid graph}. Let $M$ be a matching of $G$ such that $E_d^{j,j+1}(G)\subseteq M$, and let $F\subseteq E(G)$ such that $f\in C\cap F\subseteq M\cap F$. Then $C$ is an $(F,M)$-nice cycle. Here $C\cap F$ has one or two $F$-fault edges. Let $f'=u'v'$. If $f'\notin F$, then $C\cap F=\{f\}$. If $f'\in F$, then $C\cap F=\{f,f'\}$.

\begin{lemma}\label{lem disjoint nice cycles no F}
Let $G$ be a graph, let $F\subseteq E(G)$, and let $M\subseteq E(G)$ be a matching of $G$ such that $F\cap M\neq \emptyset$. Let $s\geqslant 1$ and let $C_1,C_2,\ldots,C_s$ be edge-disjoint $(F,M)$-nice cycles such that $F\cap M\subseteq \bigcup\limits_{i=1}^sC_i$. Let $M_\Delta=M\Delta C_1\Delta C_2\Delta \cdots \Delta C_s$. Then $|M_\Delta|=|M|$ and $M_\Delta\cap F=\emptyset$, i.e. $M_\Delta$ contains no $F$-fault edges. In particular, if $M$ is a perfect matching (or almost perfect matching, respectively), then $M_\Delta$ is a perfect matching (or almost perfect matching, respectively) of $G-F$.
\end{lemma}

\begin{proof}
Since $C_1,C_2,\ldots,C_s$ are edge-disjoint $M$-alternating cycles, the matching $M_\Delta$ has a partition $M_\Delta=\bigcup\limits_{i=0}^sN_i$ where $N_0=M\setminus C$, $C=\bigcup\limits_{i=1}^sC_i$ and $N_i=C_i\setminus M$ for each $1\leqslant i\leqslant s$. Now suppose $f\in M_\Delta\cap F=\bigcup\limits_{i=0}^sN_i\cap F$. Then $f\in N_i\cap F$ for some $1\leqslant i\leqslant s$. If $i=0$, then $f\in M$ and $f\notin \bigcup\limits_{i=1}^sC_i$. So $f\in M\cap F\subseteq \bigcup\limits_{i=1}^sC_i$. This is a contradiction. If $1\leqslant i\leqslant s$, then $f\in C_i$ and $f\notin M$. So $f\in C_i\cap F\subseteq M\cap F\subseteq M$, which is a contradiction. Thus $M_\Delta\cap F=\emptyset$.
\end{proof}

The following corollary follows directly by Lemma~\ref{lem disjoint nice cycles no F}.

\begin{corollary}\label{cor not mp set F}
Let $G$ be a graph, let $F\subseteq E(G)$. If there exists a perfect matching (or almost perfect matching, respectively) $M$ of $G$ and edge-disjoint $(F,M)$-nice cycles $C_1,C_2,\ldots,C_s$ such that $F\cap M\subseteq \bigcup\limits_{i=1}^sC_i$, then $F$ is not a matching preclusion set of $G$.
\end{corollary}

Let $G$ be a graph, let $F\subseteq E(G)$, and let $M\subseteq E(G)$ be a matching of $G$ such that $F\cap M\neq \emptyset$. Now we divide edges of $F\cap M$ into two classes through $(F,M)$-nice cycles. Let $f\in F\cap M$ be an $F$-fault edge. If there exists an $(F,M)$-nice cycle $C$ such that $f\in C$, then the $F$-fault edge $f$ is called an \emph{$(F,M)$-nice-fault edge}. Let $g\in F\cap M$ be an $F$-fault edge. If there has no $(F,M)$-nice cycles containing $g$, then the fault edge $g$ is called an \emph{$(F,M)$-bad-fault edge}.

\subsection{Even order $n$-grid graphs with $n\geqslant 3$}\label{subsection $n$-Grid graph for n geqslant 3}

Let $G$ be a $(k_0,k_1,\ldots,k_{n-1};n)$-grid graph. Let $n_{e}(G)$ denote the number of even numbers in $k_{0},k_{1},\ldots,k_{n-1}$. Then $1\leqslant n_{e}(G) \leqslant n$ if the order $n(G)$ is even. First, we deal the simplest case that when $n_{e}(G)=1$.

\begin{lemma}\label{lem one even number edge set not mp set}
Let $n\geqslant 3$ and let $G$ be a $(k_0,k_1,\ldots,k_{n-1};n)$-grid graph where $k_0\geqslant 2$ is even and $k_j\geqslant 3$ is odd for $1\leqslant j\leqslant n-1$. Let $M_{0}=E_{0}^{0,1}(G)\cup E_{0}^{2,3}(G)\cup\cdots\cup E_{0}^{k_{0}-2,k_{0}-1}(G)$ and let $F\subseteq E(G)$ such that $|F|\leqslant n$, $F\cap M_0\neq\emptyset$ and $d_F(f)\leqslant n-2$ for $f\in F\cap M_0$. Then $G-F$ has a perfect matching.
\end{lemma}

\begin{proof}
For $0\leqslant j\leqslant k_0-1$, let $F_j=F\cap E(G_0[j])$. Since $F\cap M_0\neq\emptyset$, then $|F_j|\leqslant n-1$. First let $|F_j|\leqslant n-2$ for any $0\leqslant j\leqslant k_0-1$. Since the size $|V_{allEven}(G_0[j])|=\frac{k_1+1}{2}\frac{k_2+1}{2}\cdots \frac{k_{n-1}+1}{2}\geqslant 2^{n-1} > n$, we can choose a vertex $u_j\in V_{allEven}(G_0[j])$ such that $P=u_0u_1\cdots u_{k_0-1}$ is a path of length $k_0-1$ such that $F$ uncovers $P$. So $|F_j| < n-1 = mp(G_0[j]-u_j)$ and $G_0[j]-F_j$ has an almost perfect matching $M[j]$ which uncovers $u_j$. Let $M_P$ be an almost perfect matching of $P$. Then $M[1]\cup M[2]\cup\cdots \cup M[k_0-1]\cup M_P$ is an almost perfect matching of $G-F$.

Now let $|F_j|=n-1$ for some $0\leqslant j\leqslant k_0-1$. Then $|F\setminus F_j|\leqslant 1$. Since $F\cap M_0\neq\emptyset$, we have $\emptyset \neq F\setminus F_j\subseteq M_0$. Let $\{f\}=F\setminus F_j$. Then $f\in F\cap M_0$. If $G_0[j]-F_j$ has an isolated vertex $x$, then all the edges in $F_j$ are incident with $x$, $d_{G_0[j]}(x)=n-1$ and $x\in V_\delta(G_0[j])\subseteq V_{allEven}(G_0[j])$. If $x$ is incident with $f$, then $d_F(f)=n-1>n-2$ which is a contradiction. So $x$ is not incident with $f$, and there exists an $(f;4)$-cycle $C$ such that $F_j\cap C=\emptyset$. Thus $M_0\Delta C$ is a perfect matching of $G-F$. Suppose $G_0[j]-F_j$ has no isolated vertices. Then there exists an $(f;4)$-cycle $C'$ such that $F_j\cap C'=\emptyset$. Hence $M_0\Delta C'$ is a perfect matching of $G-F$.
\end{proof}

\begin{theorem}\label{thm one even PM set}
Let $n\geqslant 3$ and let $G$ be an $(k_0,k_1,\ldots,k_{n-1};n)$-grid graph. If $n_{e}(G)=1$, then $G$ is super matched.
\end{theorem}

\begin{proof} Let $F\subseteq E(G)$ be an optimal matching preclusion set of $G$, i.e. $|F|=n$ and $G-F$ has no perfect matching. We need to prove that $F$ is trivial, i.e. $G-F$ has an isolated vertex. Without loss of generality, we may assume $k_{0}$ is even. Then $k_j\geqslant 3$ is odd for $1\leqslant j\leqslant n-1$. Let $M_{0}=E_{0}^{0,1}(G)\cup E_{0}^{2,3}(G)\cup\cdots\cup E_{0}^{k_{0}-2,k_{0}-1}(G)$. Then $M_0$ is a perfect matching of $G$ and $M_{0}\cap F\neq \emptyset$. If $d_F(f)\leqslant n-2$ for any $f\in F\cap M_0$, then by Lemma~\ref{lem one even number edge set not mp set}, $G-F$ has a perfect matching. This is a contradiction. So there exists $f\in F\cap M_0$ such that $d_F(f)=n-1$. Let $f$ be such an edge. Then $F\cap M_0=\{f\}$ and other $F$-fault edges are incident with $f$. We assume that $f=u_{j}u_{j+1}$ for some $u_{j}\in V(G_{0}[j])$ and $u_{j+1}\in V(G_{0}[j+1])$, where $j$ is even and $0\leqslant j \leqslant k_0-2$.

For $i=j$ or $j+1$, let $F_{i}=F\cap E(G_{0}[i])$, and $P_i=\{d\mid d$ is the position of $e$, $e\in F_i\}$. Then $F_j\cap F_{j+1}=\emptyset$, $F_j\cup F_{j+1}\subseteq F\setminus \{f\}$, $|P_j|\leqslant |F_j|$, $|P_{j+1}|\leqslant |F_{j+1}|$, $0\not\in P_j\cup P_{j+1}$ and $P_j\cup P_{j+1}\subseteq \{1,2,\ldots,n-1\}$. If $P_j\cup P_{j+1}\neq \{1,2,\ldots,n-1\}$, then for any position $d\in \{1,2,\ldots,n-1\}\setminus (P_j\cup P_{j+1})$, we can find an $(f,0,j,d;4)$-cycle $C$ such that $M_0\Delta C$ is a perfect matching of $G-F$. So $P_j\cup P_{j+1}=\{1,2,\ldots,n-1\}$, $P_j\cap P_{j+1}=\emptyset$, $|P_j|=|F_j|$, $|P_{j+1}|=|F_{j+1}|$ and $F_j\cup F_{j+1}=F\setminus \{f\}$. If $F_{j}\neq \emptyset$ and $F_{j+1}\neq \emptyset$, then $|F_{j}|,|F_{j+1}|\leqslant n-2$. Then by the same argument as in the first paragraph of the proof of Lemma~\ref{lem one even number edge set not mp set}, $G-F$ has a perfect matching. So one of $F_{j}$ and $F_{j+1}$ is empty.

Without loss of generality, we may assume that $F_{j+1}=\emptyset$. Then $P_j=\{1,2,\ldots,n-1\}$, $F_j=F\setminus \{f\}$, and all the edges in $F$ are incident with $u_{j}$. If $d_{G_{0}[j]}(u_{j}) > n-1$, then for some position $d\in \{1,2,\ldots,n-1\}$, there are two edges incident with $u_j$ at position $d$, in which one is not in $F$. So we can find an $(f,0,j,d;4)$-cycle $C'$ such that $M_0\Delta C'$ is a perfect matching of $G-F$. Hence $d_{G_{0}[j]}(u_{j}) = n-1$ and $u_{j}\in V_\delta(G_0[j])\subseteq V_{allEven}(G_0[j])$ is an isolated vertex of $G_{0}[j]-F_j$.

\newsavebox{\evenorderoneMI}
\savebox{\evenorderoneMI}{
\setlength{\unitlength}{1em}
\begin{picture}(20,10)\label{pic evenorderoneMI}

\put(2.5,5){\oval(3,6)}
\put(7.5,5){\oval(3,6)}
\put(12.5,5){\oval(3,6)}
\put(17.5,5){\oval(3,6)}

\put(0.5,0.5){$G_0[j-2]$}
\put(5.5,0.5){$G_0[j-1]$}
\put(11.5,0.5){$G_0[j]$}
\put(15.5,0.5){$G_0[j+1]$}

\put(1.5,8.5){$M^{j-2}$}
\put(6.5,8.5){$M^{j-1}$}
\put(12,8.5){$M^j$}
\put(16.5,8.5){$M^{j+1}$}

\put(2.5,7){\circle*{0.3}}
\put(7.5,3.5){\circle*{0.3}}
\put(7.5,4.5){\circle*{0.3}}
\put(7.5,7){\circle*{0.3}}
\put(12.5,3.5){\circle*{0.3}}
\put(12.5,4.5){\circle*{0.3}}
\put(12.5,7){\circle*{0.3}}
\put(17.5,3.5){\circle*{0.3}}
\put(17.5,7){\circle*{0.3}}

\put(2.5,7){\line(1,0){15}}
\put(7.5,4.5){\line(1,0){5}}
\put(7.5,3.5){\line(1,0){10}}
\put(7.5,3.5){\line(0,1){1}}
\put(12.5,3.5){\line(0,1){1}}

\put(1.5,6){$x_{j-2}$}
\put(6.5,6){$x_{j-1}$}
\put(12,6){$x_{j}$}
\put(16.5,6){$x_{j+1}$}

\put(6.5,5){$w_{j-1}$}
\put(12,5){$w_{j}$}

\put(6.5,2.5){$u_{j-1}$}
\put(12,2.5){$u_{j}$}
\put(16.5,2.5){$u_{j+1}$}

\put(14.75,2.5){$f$}
\put(9.75,2.75){$g$}
\put(9.75,4.75){$h$}
\end{picture}}

\begin{figure}[ht]
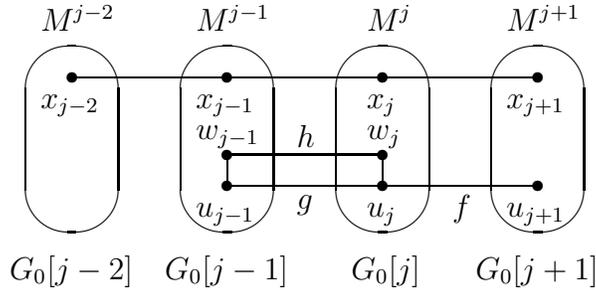

\centering
\usebox{\evenorderoneMI}
\caption{The construction of $M_I$.}
\label{fig: evenorderoneMI}
\end{figure}

If $k_0=2$, then $j = 0$, $d_{G}(u_{j}) = n$, and $u_{j}$ is an isolated vertex of $G-F$. Now we assume that $k_0\geqslant 4$. Suppose $j > 0$. Then $2\leqslant j\leqslant k_0-2$. Let $I=\{j-2,j-1,j,j+1\}$. Then we can find a perfect matching $M_F=M_0'\cup M_I$ of $G-F$. The following is the construction of $M_F$. Let $M_0'=M_{0}\setminus (E_{0}^{j-2,j-1}(G)\cup E_{0}^{j,j+1}(G))$, then $M_0'$ is a matching of $G$ such that $M_0'$ uncovers subgraphs $G_0[j-2],G_0[j-1],G_0[j],G_0[j+1]$. These four subgraphs and the construction of $M_I$ are shown in Figure~\ref{fig: evenorderoneMI}. For each $i\in I$, let $x_i\in V_{allEven}(G_0[i])$ such that $x_{j-2}x_{j-1}x_{j}x_{j+1}$ is a path whose edges are at position $0$. By Lemma~\ref{lem Alleven apm}, $G_0[i]$ has an almost perfect matching $M^i$ such that $M^i$ uncovers $x_i$. Since $n\geqslant 3$ and $k_i\geqslant 3$ for each $i\in \{1,2,\ldots,n-1\}$, we have $|V_\delta(G_0[j])|=2^{n-1}\geqslant 4>1$ and we can choose $x_j\in V_\delta(G_0[j])\subseteq V_{allEven}(G_0[j])$ such that $x_j\neq u_j$ and so $x_j$ is not adjacent to $u_j$. Then $x_{j+1}\neq u_{j+1}$ and $x_{j+1}$ is not adjacent to $u_{j+1}$. Let $u_jw_j\in M^j$ for some $w_j\in V(G_0[j])$. Let $g=u_ju_{j-1}$ and $h=w_jw_{j-1}$ be edges at position $0$ where $u_{j-1},w_{j-1}\in V(G_0[j-1])$. Then we can choose $M^{j-1}$ (a copy of $M^j$) such that $u_{j-1}w_{j-1}\in M^{j-1}$. Let $M_I'=M^{j-2}\cup \biggl(M^{j-1}\setminus \{u_{j-1}w_{j-1}\}\biggr)\cup \biggl(M^{j}\setminus \{u_{j}w_{j}\}\biggr)\cup M^{j+1}$ and let $S=\{x_{j-2},x_{j-1},x_j,x_{j+1},u_{j-1},u_j,w_{j-1},w_j\}$. Then $M_I'\cap F=\emptyset$ and $M_I'$ is a matching of $G$ such that $M_I'$ covers vertices in $V(G_0[i])\setminus S$ for each $i\in I$. Let $M_I=M_I'\cup \{g,h\}\cup \{x_{j-2}x_{j-1},x_{j}x_{j+1}\}$. Then $M_F=M_0'\cup M_I$ is a perfect matching of $G-F$. Hence $j = 0$. So $d_{G}(u_{j}) = n$ and $u_{j}$ is an isolated vertex of $G-F$.
\end{proof}

\begin{theorem}\label{thm even PM set and number}
Let $n\geqslant 3$ and let $G$ be an even order $(k_0,k_1,\ldots,k_{n-1};n)$-grid graph. Then $G$ is super matched.
\end{theorem}

\begin{proof}
Without loss of generality, we may assume $k_{0}$ is even. Let $M=E_{0}^{0,1}(G)\cup E_{0}^{2,3}(G)\cup \cdots\cup E_{0}^{k_{0}-2,k_{0}-1}(G)$. Then $M$ is a perfect matching of $G$. Here we prove that $G$ is super matched by induction on $n_{e}(G)$. When $n_{e}(G)=1$, the statement holds by Theorem~\ref{thm one even PM set}. Now we suppose $n_{e}(G)\geqslant 2$. Let $F$ be an optimal matching preclusion set of $G$, i.e. $|F|=n$ and $G-F$ has no perfect matching. Then $M\cap F\neq \emptyset$. We need to show that $F$ is trivial, i.e. $G-F$ has an isolated vertex.

Let $f=x_{k}x_{k+1}\in M\cap F$ for some $x_{k}\in V(G_{0}[k])$, $x_{k+1}\in V(G_{0}[k+1])$, where $k$ is even and $0\leqslant k\leqslant k_0-2$. If for each $0\leqslant s\leqslant k_0-1$, $G_{0}[s]-F$ has a perfect matching $M_{s}$, then $M_{0}\cup M_{1}\cup \cdots\cup M_{k_{0}-1}$ is a perfect matching of $G-F$. So there exists $t\in \{0,1,\ldots,k_{0}-1\}$ such that $G_{0}[t]-F$ has no perfect matching, i.e. $F_{t}=F\cap E(G_{0}[t])\subseteq F\setminus \{f\}$ is a matching preclusion set of $G_{0}[t]$. The subgraph $G_0[t]$ is an even order $(k_{1},\ldots,k_{n-1};n-1)$-grid graph. Let $n_{e}(G_0[t])$ be the number of even numbers in $k_{1},\ldots,k_{n-1}$. Then $1\leqslant n_e(G_0[t])<n_e(G)$. By Lemma~\ref{lem mp number even order}, we have $mp(G_0[t])=n-1$. So $|F_t|\geqslant n-1$. Since $|F_t|\leqslant n-1$, we have $|F_t|= n-1$ and $F_t$ is an optimal matching preclusion set of $G_0[t]$. By induction hypothesis, the optimal matching preclusion set $F_t$ is trivial, i.e. $G_0[t]-F$ has an isolated vertex $u$ with vertex degree $d_{G_{0}[t]}(u)=n-1$.

Now for each $j\in \{0,1,\ldots,k_{0}-1\}\setminus \{t\}$, $F_j=E(G_{0}[j])\cap F=\emptyset$ and $G_{0}[j]-F_j$ has a perfect matching $M[j]$. If $u \notin \{x_{k},x_{k+1}\}$, then we can find an $(f;4)$-cycle $C$ such that $M\Delta C$ is a perfect matching of $G-F$. So $u\in\{x_{k},x_{k+1}\}$. Without loss of generality, we may assume $u=x_k$, then $t=k$. If $k\neq 0$, then $M[0]\cup M[1]\cup\ldots\cup M[k-2]\cup E_{0}^{k-1,k}(G)\cup M[k+1]\cup\ldots\cup M[k_{0}-1]$ is a perfect matching of $G-F$. This is a contradiction. Thus $k=t=0$. So $d_{G}(u)=n$ and $u$ is an isolated vertex of $G-F$.
\end{proof}

Lemma~\ref{lem mp number even order} and Theorem~\ref{thm even PM set and number} imply the even order part of our main theorem.

\begin{theorem}\label{thm even order mp}
Let $G$ be an even order $n$-grid graph. Then $mp(G)=n$. If $n\geqslant 3$, then $G$ is super matched.
\end{theorem}

In particular, the $(2,2,\ldots,2;n)$-grid graph is the hypercube $Q_{n}$. Lemmas~\ref{lem mp number even order} and \ref{lem 2-grid PM set} and Theorem~\ref{thm even PM set and number} imply immediately the following result.

\begin{corollary}[\cite{3}]
$mp(Q_{n})=n$ and $Q_{n}$ is super matched.
\end{corollary}


\section{Matching preclusion for odd order $n$-grid graphs}\label{section Matching preclusion of odd order $n$-grid graphs}

Let $G$ be an odd order $(k_0,k_1,\ldots,k_{n-1};n)$-grid. Then $k_i\geqslant 3$ is odd for each $0\leqslant i \leqslant n-1$. By using results of matching preclusion for even order $n$-grid graphs, we get matching preclusion for odd order $n$-grid graphs in this section.

By Lemma~\ref{lem Sumeven apm}, for any vertex $u\in V_e(G)$, we can find an almost perfect matching $M_u$ of $G$ which uncovers $u$. By Lemma~\ref{lem Bipartite partition}, for any vertex $u\in V_o(G)$, the graph $G-u$ has no perfect matching. Then the set of edges incident with a vertex in $V_{o}(G)$ is a matching preclusion set. In the proof of Theorem~\ref{thm odd PM number}, we may follow the summary of this proof in Figure~\ref{fig: oddordertable}.

\newsavebox{\oddordertable}
\savebox{\oddordertable}{
$mp(G)>n \left\{\begin{tabular}{l}$n=1$\\ \hline\\ $n\geqslant 2$
$\left\{\begin{tabular}{l}$G-F$ has an isolated vertex $u$\\ \hline\\ $G-F$ has no isolated vertex\\ $\Downarrow$\\
$\left\{\begin{tabular}{l}$F_{j2}\neq \emptyset$ for all $0\leqslant j\leqslant n-1$\\ \hline\\ $F_{i2}=\emptyset$ for some $0\leqslant i\leqslant n-1$\\ $\Downarrow$\\
$\left\{\begin{tabular}{l}$|F_{i1}|<n$\\ $|F_{i1}|=n$
$\left\{\begin{tabular}{l}$H_{i1}-F$ has an islolated vertex $v$\\  \hline\\ $H_{i1}-F$ has no islolated vertex\\ $\Downarrow$\\
$\left\{\begin{tabular}{l}$n\geqslant 3$\\ $n=2$
\end{tabular}\right.$
\end{tabular}\right.$
\end{tabular}\right.$
\end{tabular}\right.$
\end{tabular}\right.$
\end{tabular}\right.$
}

\begin{figure}[ht]
\centering
\usebox{\oddordertable}
\caption{Summary of the proof of Theorem~\ref{thm odd PM number}.}
\label{fig: oddordertable}
\end{figure}

\begin{theorem}\label{thm odd PM number}
Let $G$ be an odd order $(k_0,k_1,\ldots,k_{n-1};n)$-grid graph. Then $mp(G)=n+1$.
\end{theorem}

\begin{proof}
By Lemma~\ref{lem Bipartite partition}, $mp(G)\leqslant \min\{ d_{G}(v)~|~v\in V_{o}(G)\}=n+1$. Next by induction on $n$, we will show that $mp(G)>n$, i.e. for any $F\subseteq E(G)$ with $|F|\leqslant n$, $G-F$ has an almost perfect matching.

If $n=1$, $G$ is the path $P_{k_0}$ where $k_0\geqslant 3$ is odd. Then for any $F\subset E(G)$ with $|F|\leqslant 1$, it is easy to find an almost perfect matching in $G-F$.

Now suppose $n\geqslant 2$. Let $F\subseteq E(G)$ such that $|F|\leqslant n$. First, we assume $G-F$ has an isolated vertex $u$. So $d_{G}(u)\geqslant \delta(G)=n$ and $d_{G}(u)\leqslant |F|\leqslant n$. Then $|F|=d_{G}(u)=n$, all the edges in $F$ are incident with $u$ and $u\in V_{\delta}(G)\subseteq V_{allEven}(G)$. By Lemma~\ref{lem Alleven apm}, $G$ has an almost perfect matching $M_{u}$ which uncovers $u$. Then $M_{u}$ is also an almost perfect matching of $G-F$.

Now we assume that $G-F$ has no isolated vertex. For each $0\leqslant j\leqslant n-1$, let $H_{j1}=G-G_j[k_j-1]$, $H_{j2}=G_j[k_j-1]$, $F_{j1}=F\cap E(H_{j1})$ and $F_{j2}=F\cap E(H_{j2})$. Then $H_{j1}$ is an even order $n$-grid graph, $H_{j2}$ is an odd order $(n-1)$-grid graph, the sets $V(H_{j1})$ and $V(H_{j2})$ form a partition of $V(G)$, and the sets $F_{j1}$, $F_j'=F\cap E_j^{k_j-2,k_j-1}(G)$ and $F_{j2}$ form a partition of $F$. These notations are shown in Figure~\ref{fig: oddordermpnumber}. By Lemma~\ref{lem mp number even order}, we have $mp(H_{j1})=n$. By induction hypothesis, we have $mp(H_{j2})=n$.

\newsavebox{\oddordermpnumber}
\savebox{\oddordermpnumber}{
\setlength{\unitlength}{1em}
\begin{picture}(22,11)\label{pic oddordermpnumber}

\put(2.5,6){\oval(3,6)}
\put(7.5,6){\oval(3,6)}
\put(13.5,6){\oval(3,6)}
\put(19.5,6){\oval(3,6)}

\put(0.5,1.5){\line(1,0){15}}
\put(0.5,1.5){\line(0,1){8}}
\put(0.5,9.5){\line(1,0){15}}
\put(15.5,1.5){\line(0,1){8}}

\put(10,6){$\ldots$}

\put(1.5,2){$G_j[0]$}
\put(6.5,2){$G_j[1]$}
\put(11,2){$G_j[k_j-2]$}
\put(17.5,2){$G_j[k_j-1]$}

\put(6.5,0.5){$H_{j1}$}
\put(18.5,0.5){$H_{j2}$}

\put(6.5,10){$F_{j1}$}
\put(16,10){$F_j'$}
\put(18.5,10){$F_{j2}$}
\end{picture}}

\begin{figure}[ht]
\centering
\usebox{\oddordermpnumber}
\caption{The partition of the odd order $n$-grid graph $G$ at position $j$.}
\label{fig: oddordermpnumber}
\end{figure}

Suppose $F_{j2}\neq \emptyset$ for all $0\leqslant j\leqslant n-1$. Let $f\in F$ and let $d$ be the position of $f$. Then $f\notin F_{d2}$, $F_{d2}\neq \emptyset$ and $F_{d1}\cup F_{d2}\subseteq F$. So $|F_{d1}|<n=mp(H_{d1})$ and $|F_{d2}|<n=mp(H_{d2})$. For $s=1$ or $2$, $F_{ds}$ is not a matching preclusion set of $H_{ds}$ and $H_{ds}-F_{ds}$ has a perfect matching $M_{ds}$. Thus $M=M_{d1}\cup M_{d2}$ is an almost perfect matching of $G-F$.

The rest case is that there exists an $i\in \{0,1,\ldots,n-1\}$ such that $F_{i2}=\emptyset$. When $|F_{i1}|<n=mp(H_{i1})$, the set $F_{i1}$ is not a matching preclusion set of $H_{i1}$. So $H_{i1}-F_{i1}$ has a perfect matching $M_{i1}'$. Let $M_{i2}'$ be an almost perfect matching of $H_{i2}=H_{i2}-F_{i2}$. Then $M_{i1}'\cup M_{i2}'$ is an almost perfect matching of $G-F$.

When $|F_{i1}|=n$, we have $F=F_{i1}\subseteq E(H_{i1})$. If $H_{i1}-F$ has an isolated vertex $v$, then $d_{H_{i1}}(v)=n$. Since $G-F$ has no isolated vertex, the vertex $v\in G_i[k_i-2]$, $d_{G_i[k_i-2]}(v)=n-1$. So $v\in V_\delta (G_i[k_i-2])\subseteq V_{allEven}(G_i[k_i-2])$. By Lemma~\ref{lem Alleven apm}, $G_i[k_i-2]$ has an almost perfect matching $M_v$ which uncovers $v$. Let $v^+\in V(H_{i2})$ such that $vv^+\in E_i^{k_i-2,k_i-1}(G)$. Then $vv^+\notin F$ and $v^+\in V_{allEven}(H_{i2})$. By Lemma~\ref{lem Alleven apm}, $H_{i2}$ has an almost perfect matching $M_v^+$ such that $M_v^+$ uncovers $v^+$. Let $M_v^-$ be any almost perfect matching of $H_{i1}-G_i[k_i-2]$. Then $M_v^-\cup M_v\cup M_v^+\cup \{vv^+\}$ is an almost perfect matching of $G-F$.

Now we assume $H_{i1}-F$ has no isolated vertex. When $n\geqslant 3$, $F$ is not a matching preclusion set of $H_{i}$ by Theorem~\ref{thm one even PM set}. So $G-F$ has an almost perfect matching. When $n=2$, without loss of generality, we may assume $i=0$. Then $G$ is a $(k_0,k_1;2)$-grid graph and $H_{01}$ is an even order $(k_0-1,k_1;2)$-grid graph where $k_0\geqslant 3$ and $k_1\geqslant 3$ are odd integers. If $F$ is not a matching preclusion set of $H_{01}$, then $G-F$ has an almost perfect matching. Suppose $F$ is a matching preclusion set of $H_{01}$. Since $H_{01}-F$ has no isolated vertex, by Lemma~\ref{lem 2-grid PM set}, we have $k_1=3$ and the endpoints of edges in $F$ are $(u_0,0)$, $(u_0+1,0)$, $(u_0,2)$ and $(u_0+1,2)$ for some even number $u_0\in \{0, 2,\ldots, k_0-3\}$. Then $F\subseteq E_0^{u_0,u_0+1}(G)$. Let $g\in E(G_0[0])$. Then $\{g\}\cup E_0^{1,2}(G)\cup E_0^{3,4}(G)\cup \cdots \cup E_0^{k_0-2,k_0-1}(G)$ is an almost perfect matching of $G-F$.
\end{proof}

Lemma~\ref{lem mp number odd order n>=3}, in Appendix~\ref{appendix sec mp number odd order n>=2}, gives another proof of Theorem~\ref{thm odd PM number} when $n\geqslant 3$ without using results of matching preclusion for even order $n$-grid graphs.

Next we will discuss the structure of the optimal matching preclusion set of  odd order $n$-grid graphs. The $(k_0;1)$-grid graph $G$ is a path $P_{k_0}$ with  $k_0$  odd. Then $mp(G)=2$. Let $F\subseteq E(G)$ be an optimal matching preclusion set. Then $|F|=2$ and $P_{k_0}-F$ is isomorphic to the disjoint union $P_s\cup P_t\cup P_q$ of three paths where the positive numbers $s,t,q$ are odd and $s+t+q=k_0$.

\begin{theorem}\label{thm odd PM set}
Let $n\geqslant 2$ and let $G$ be an odd order $(k_0,k_1,\ldots,k_{n-1};n)$-grid graph. Then an edge set $F\subseteq E(G)$ is an optimal matching preclusion set of $G$ if and only if $F$ consists of edges incident with a vertex $u\in V_{o}(G)$ with $d_G(u)=n+1$.
\end{theorem}

\begin{proof}
Let $F\subseteq E(G)$ be an optimal matching preclusion set of $G$. Then $|F|=n+1$ and $G-F$ has no almost perfect matching. Let $f=xy\in F$, let $d$ be the position of $f$ and let $x\in G_{d}[j]$ and $y\in G_{d}[j+1]$ for some $j\in \{0,1,\ldots,k_{d}-2\}$. Let $H_{1}$ and $H_{2}$ be the connected components of $G-E_{d}^{j,j+1}(G)$ such that $x\in V(H_{1})$ and $y\in V(H_{2})$. Then one of $H_{1}$ and $H_{2}$ is of even order and the other is of odd order. Without loss of generality, we may assume $H_{1}$ is of odd order. Then $j \leqslant k_{d}-3$ is even and $H_{2}$ is of even order. Then $mp(H_2)=n$ by Lemma~\ref{lem mp number even order}. Let $F_1=F\cap E(H_1)$ and $F_2=F\cap E(H_2)$. Then $F_1\cap F_2=\emptyset$, $F_1\cup F_2 \subseteq F\setminus \{f\}$, and $|F_{1}|, |F_{2}|\leqslant n$. The above partition at position $d$ is shown in Figure~\ref{fig: oddordermpset}. Let $M_{0}=E_{d}^{0,1}(G)\cup E_{d}^{2,3}(G)\cup \cdots \cup E_{d}^{k_{d}-3,k_{d}-2}(G)$. Then $M_0$ is a matching of $G$ and $f\in M_0$.

\newsavebox{\oddordermpset}
\savebox{\oddordermpset}{
\setlength{\unitlength}{1em}
\begin{picture}(33,11)\label{pic oddordermpset}

\put(2.5,6){\oval(3,6)}
\put(7.5,6){\oval(3,6)}
\put(10,6){$\ldots$}
\put(13.5,6){\oval(3,6)}
\put(19.5,6){\oval(3,6)}
\put(22,6){$\ldots$}
\put(25.5,6){\oval(3,6)}
\put(30.5,6){\oval(3,6)}

\put(0.5,1.5){\line(1,0){15}}
\put(0.5,1.5){\line(0,1){8}}
\put(0.5,9.5){\line(1,0){15}}
\put(15.5,1.5){\line(0,1){8}}

\put(17.5,1.5){\line(1,0){15}}
\put(17.5,1.5){\line(0,1){8}}
\put(17.5,9.5){\line(1,0){15}}
\put(32.5,1.5){\line(0,1){8}}

\put(1.5,2){$G_d[0]$}
\put(6.5,2){$G_d[1]$}
\put(12.5,2){$G_d[j]$}

\put(18,2){$G_d[j+1]$}
\put(23,2){$G_d[k_d-2]$}
\put(28,2){$G_d[k_d-1]$}

\put(13.5,6){\circle*{0.3}}
\put(19.5,6){\circle*{0.3}}

\put(13.5,6){\line(1,0){6}}

\put(13.25,5){$x$}
\put(19.25,5){$y$}
\put(16,6.5){$f$}

\put(8,0.5){$H_1$}
\put(23,0.5){$H_2$}

\put(8,10){$F_1$}
\put(23,10){$F_2$}
\end{picture}}

\begin{figure}[ht]
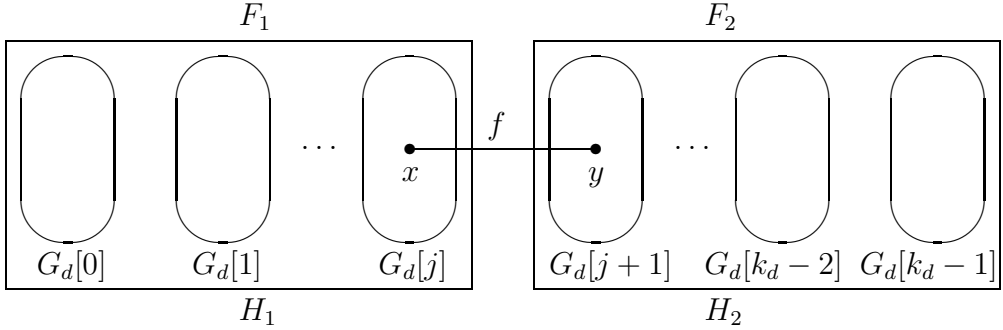

\centering
\usebox{\oddordermpset}
\caption{The partition of the odd order $n$-grid graph $G$ at position $d$.}
\label{fig: oddordermpset}
\end{figure}

If $F_{2}$ is a matching preclusion set of $H_{2}$, then $|F_{2}|\geqslant n$. So $|F_{2}|=n$. By Theorem~\ref{thm one even PM set}, $H_{2}-F_{2}$ has an isolated vertex $v$, all the $F_2$-fault edges are incident with $v$ and $d_{H_{2}}(v)=n$. Then $v$ is a vertex of $G_{d}[j+1]$ or $G_{d}[k_{d}-1]$. If $v\in V(G_{d}[k_{d}-1])$. We have $d_{G}(v)=d_{H_{2}}(v)=n$, $v\in V_{\delta}(G)\subseteq V_{allEven}(G)$ and $v\in V_{allEven}(G_{d}[k_{d}-1])$. By Lemma~\ref{lem Alleven apm}, $G_{d}[k_{d}-1]$ has an almost perfect matching $M_{v}$ which uncovers $v$. There exists an $(f;4)$-cycle $C_f^1$ such that $C_f^1$ is an $(F,M_{0})$-nice cycle. Then $(M_{0}\Delta C_{f}^{1})\cup M_{v}$ is an almost perfect matching of $G-F$. So $v\in V(G_{d}[j+1])$. If $y\neq v$, then there exists an $(f;4)$-cycle $C_f^2$ such that $C_f^2$ is an $(F,M_{0})$-nice cycle. Let $M_{v}^{k_{d}-1}$ be an almost perfect matching of $G_{d}[k_{d}-1]$. Then $(M_{0}\Delta C_{f}^{2})\cup M_{v}^{k_{d}-1}$ is an almost perfect matching of $G-F$. So $y=v$ and all the $F$-fault edges are incident with $v$. Since $d_{H_{2}}(v)=n$, the vertex degree $d_{G}(v) = n+1=|F|$. Hence $v$ is an isolated vertex of $G-F$. Suppose $v\in V_{e}(G)$. Then by Lemma~\ref{lem Sumeven apm}, $G$ has an almost perfect matching $M'_v$ which uncovers $v$. So $M'_v$ is also an almost perfect matching of $G-F$. Thus $v\in V_{o}(G)$.

If $F_{2}$ is not a matching preclusion set of $H_{2}$, then $H_{2}-F_{2}$ has a perfect matching $M_{2}$. If $F_{1}$ is not a matching preclusion set of $H_{1}$. Then $H_{1}-F_{1}$ has an almost perfect matching $M_{1}$. So $M_{1}\cup M_{2}$ is an almost perfect matching of $G-F$. This is a contradiction. So $F_{1}$ is a matching preclusion set of $H_{1}$. The matching preclusion number $mp(H_{1})=n$ or $n+1$ by Theorem~\ref{thm odd PM number}. Then $|F_{1}|\geqslant mp(H_{1})\geqslant n$. Since $|F_{1}|\leqslant n$, we have $|F_{1}|=mp(H_{1}) = n$. Then $H_{1}$ is an odd order $(n-1)$-grid graph, $j=0$, and $F_{2}=\emptyset$. Let $M^{k_d-1}$ be an almost perfect matching of $G_{d}[k_{d}-1]$. If $x$ is not an isolated vertex in $H_1-F_1$, then there exists an $(f,d,0;4)$-cycle $C_f^3$ such that $C_f^3$ is an $(F,M_0)$-nice cycle. So $(M_{0}\Delta C_f^3)\cup M^{k_d-1}$ is an almost perfect matching of $G-F$. Hence $x$ is an isolated vertex in $H_1-F_1$. The vertex degree satisfies $n-1=\delta(H_1)\leqslant d_{H_1}(x)\leqslant |F_1|=n$. If $d_{H_1}(x)=n-1$, then $d_G(x)=n$ and there exists exactly one edge $g\in F_1$ which is not incident with $x$. Since $x\in V_\delta(G)\subseteq V_{allEven}(G)$ and $|\{g\}|=1<n$, the set $\{g\}$ is not a matching preclusion set of $G-x$ by Lemma~\ref{lem mp(G-u) for u in VallEvenG odd order}. Then $(G-x)-g$ has a perfect matching $M^x$. This means $M^x$ is an almost perfect matching of $G-F$ which uncovers $x$. Thus $d_{H_1}(x)=n$. So $d_G(x)=n+1$, all the $F$-fault edges are incident with $x$, and $x$ is an isolated vertex of $G-F$. By the same argument as in the last paragraph, we have $x\in V_{o}(G)$.

Conversely, let $F\subseteq E(G)$ such that $|F|=n+1$ and $G-F$ has an isolated vertex $u\in V_{o}(G)$ with $d_G(u)=n+1$. Then $F$ consists of edges incident with $u$. By Lemma~\ref{lem Bipartite partition}, $F$ is an optimal matching preclusion set.
\end{proof}

Theorems~\ref{thm odd PM number} and~\ref{thm odd PM set} imply the odd order part of our main theorem.



\appendix

\section{Matching preclusion number of odd order $n$-grid graph where $n\geqslant 2$}\label{appendix sec mp number odd order n>=2}

\begin{lemma}\label{lem 2-grid mp number odd order}
Let $G$ be an odd order $(k_0,k_1;2)$-grid graph. Then $mp(G)=3$.
\end{lemma}

\begin{proof}
We follow the summary of this proof as that of Theorem~\ref{thm odd PM number} (see Figure~\ref{fig: oddordertable}). By using Lemmas~\ref{lem mp number even order} and~\ref{lem 2-grid PM set}, we can get this result.
\end{proof}

\begin{lemma}\label{lem odd order f not in apm Mf}
Let $G$ be an odd order $(k_0,k_1,\ldots,k_{n-1};n)$-grid graph and let $f\in E(G)$. Then there exists an almost perfect matching $M_f$ such that $f\not\in M_f$.
\end{lemma}

\begin{proof}
Let $d$ be the position of $f=u_ju_{j+1}$ where $u_j\in G_d[j]$ and $u_{j+1}\in G_d[j+1]$. Let $H_1$ and $H_2$ be the connected components of $H=G-E_d^{j,j+1}(G)$ where $u_j\in V(H_1)$ and $u_{j+1}\in V(H_2)$. If $j$ is even, then $H_1$ is an even order grid graph and $H_2$ is an odd order grid graph. Let $M_1$ be a perfect matching of $H_1$ and let $M_2$ be an almost perfect matching of $H_2$. Then $M=M_1\cup M_2$ is an almost perfect matching of $G$ such that $f\not\in M$. If $j$ is odd, then $H_1$ is an odd order grid graph and $H_2$ is an even order grid graph. Let $M_1'$ be an almost perfect matching of $H_1$ and let $M_2'$ be a perfect matching of $H_2$. Then $M'=M_1'\cup M_2'$ is an almost perfect matching of $G$ such that $f\not\in M'$.
\end{proof}

Let $n\geqslant 3$. Here we give a proof of matching preclusion number of odd order $n$-grid graph without using the matching preclusion results for even order grid graphs.

\begin{lemma}\label{lem mp number odd order n>=3}
Let $n\geqslant 3$ and let $G$ be an odd order $(k_0,k_1,\ldots,k_{n-1};n)$-grid graph. Then $mp(G)=n+1$.
\end{lemma}

\begin{proof}
By Lemma~\ref{lem Bipartite partition}, $mp(G)\leqslant \min\{ d_{G}(v)~|~v\in V_{o}(G)\}=n+1$. Next we will show that $mp(G)>n$, i.e. for any $F\subseteq E(G)$ with $|F|\leqslant n$, $G-F$ has an almost perfect matching. Let $F\subseteq E(G)$ such that $|F|\leqslant n$. For $0\leqslant j\leqslant k_0-1$, the subgraph $G_0[j]$ is an odd order $(n-1)$-grid graph. Let $F_j=F\cap E(G_0[j])$.

First suppose $F_j=F$ for some $0\leqslant j\leqslant k_0-1$. We consider two possibilities for $j$. If $j\neq 0$, then $M_1\cup M_{1,0}$ is an almost perfect matching of $G-F$ where $M_1=E_{0}^{1,2}(G)\cup E_{0}^{3,4}(G)\cup \cdots\cup E_{0}^{k_{0}-2,k_{0}-1}(G)$ and $M_{1,0}$ is an almost perfect matching of $G_0[0]$. If $j\neq k_0-1$, then $M_2\cup M_{2,k_0-1}$ is an almost perfect matching of $G-F$ where $M_2=E_{0}^{0,1}(G)\cup E_{0}^{2,3}(G)\cup \cdots\cup E_{0}^{k_{0}-3,k_{0}-2}(G)$ and $M_{2,k_0-1}$ is an almost perfect matching of $G_0[k_0-1]$.

Now suppose $F_j\neq F$ for any $0\leqslant j\leqslant k_0-1$. Then $|F_j|\leqslant n-1$. The following are two cases. First let $|F_j|\leqslant n-2$ for any $0\leqslant j\leqslant k_0-1$. Since the size $|V_{allEven}(G_0[j])|=\frac{k_1+1}{2}\frac{k_2+1}{2}\cdots \frac{k_{n-1}+1}{2}\geqslant 2^{n-1} > n$, we can choose a vertex $u_j\in V_{allEven}(G_0[j])$ such that $P=u_0u_1\cdots u_{k_0-1}$ is a path of length $k_0-1$ such that $F$ uncovers $P$. So $|F_j| < n-1 = mp(G_0[j]-u_j)$ and $G_0[j]-F_j$ has an almost perfect matching $M[j]$ which uncovers $u_j$ by Lemma~\ref{lem mp(G-u) for u in VallEvenG odd order}. Let $M_P$ be an almost perfect matching of $P$. Then $M[1]\cup M[2]\cup\cdots \cup M[k_0-1]\cup M_P$ is an almost perfect matching of $G-F$. The other case is that $|F_j|=n-1$ for some $0\leqslant j\leqslant k_0-1$. Let $\{f\}=F\setminus F_j$. Then $f\in E_0(G)$ or $f\in F_k$ for some $k\neq j$ where $0\leqslant k\leqslant k_0-1$.

Let $f\in F_k$ for some $k\neq j$ where $0\leqslant k\leqslant k_0-1$. Without loss of generality, we may let $j<k$. If $j$ is odd, let $H_1$ and $H_2$ be the connected components of $G-(G_0[j-1]\cup G_0[j])$ where $H_1$ is of even order, $H_2$ is of odd order and $f\in E(H_2)$. Then $H_2$ is an odd order $m$-grid graph where $m\in \{n-1,n\}$. There exists an almost perfect matching $M_f$ of $H_2$ such that $f\not\in M_f$ by Lemma~\ref{lem odd order f not in apm Mf}. Then $M(H_1)\cup E_0^{j-1,j}(G)\cup M_f$ is an almost perfect matching of $G-F$ where $M(H_1)$ is a perfect matching of $H_1$. If $j$ is even, let $H_1'$ and $H_2'$ be the connected components of $G-(G_0[j]\cup G_0[j+1])$ where $H_1'$ is of even order, $H_2'$ is of odd order and $f\in E(G_0[j+1])\cup E(H_2')$. Then $H_2'$ is an odd order $m'$-grid graph where $m'\in \{n-1,n\}$. There exists an almost perfect matching $M_f'$ of $H_2'$ such that $f\not\in M_f'$ by Lemma~\ref{lem odd order f not in apm Mf}. Then $M(H_1')\cup E_0^{j,j+1}(G)\cup M_f'$ is an almost perfect matching of $G-F$ where $M(H_1')$ is a perfect matching of $H_1'$.

Let $f\in E_0(G)$. If $G_0[j]-F_j$ has an isolated vertex $x_j$, then all the edges in $F_j$ are incident with $x_j$, $d_{G_0[j]}(x_j)=n-1$ and $x_j\in V_\delta(G_0[j])\subseteq V_{allEven}(G_0[j])$. Let $x_k\in V_{allEven}(G_0[k])$ for $k\neq j$ such that $P'=x_0x_1\cdots x_{k_0-1}$ is a path of length $k_0-1$. Then for $0\leqslant h\leqslant k_0-1$, $G_0[h]-F_h$ has an almost perfect matching $M^h$ which uncovers $x_h$ by Lemma~\ref{lem Alleven apm}. There exists an almost perfect matching $M_P'$ of $P'$ such that $f\not\in E(P')$ by Lemma~\ref{lem odd order f not in apm Mf}. Thus $M^0\cup M^1\cup \cdots \cup M^{k_0-1}\cup M_P'$ is an almost perfect matching of $G-F$.

Suppose $G_0[j]-F_j$ has no isolated vertices. Let $M_1$, $M_{1,0}$, $M_2$ and $M_{2,k_0-1}$ be defined above. First let $j\neq 0$. If $f\not\in M_1$, then $M_1\cup M_{1,0}$ is an almost perfect matching of $G-F$. If $f\in M_1$, then there exists an $(f;4)$-cycle $C_1$ such that $C_1\cap F_j=\emptyset$. So $(M_1\Delta C_1)\cup M_{1,0}$ is an almost perfect matching of $G-F$. Now let $j\neq k_0-1$. If $f\not\in M_2$, then $M_2\cup M_{2,k_0-1}$ is an almost perfect matching of $G-F$. If $f\in M_2$, then there exists an $(f;4)$-cycle $C_2$ such that $C_2\cap F_j=\emptyset$. So $(M_2\Delta C_2)\cup M_{2,k_0-1}$ is an almost perfect matching of $G-F$.
\end{proof}


\begin{thebibliography}{99}\setlength{\itemsep}{0mm}\linespread{1}\selectfont
\bibitem{1} A.~R.~de~Almeida, F.~Protti, and L.~Markenzon, Matching preclusion number in Cartesian product of graphs and its application to interconnection networks, Ars Combin., 112:193-204, 2013.

\bibitem{2} R.~Bhaskar, E.~Cheng, M.~Liang, S.~Pandey, and K.~Wang, Matching preclusion and conditional matching preclusion problems for twisted cubes, Congr. Numer., 205:175-185, 2010.

\bibitem{3} R.~C.~Brigham, F.~Harary, E.~C.~Violin, and J.~Yellen, Perfect-matching preclusion, Congr. Numer., 174:185-192, 2005.

\bibitem{5} E.~Cheng, P.~Hu, R.~Jia, and L.~Lipt\'{a}k, Matching preclusion and conditional matching preclusion for bipartite interconnection networks \textrm{II}: Cayley graphs generated by transposition trees and hyper-stars, Networks, 59(4):357-364, 2012.

\bibitem{6} E.~Cheng, P.~Hu, R.~Jia, and L.~Lipt\'{a}k, Matching preclusion and conditional matching preclusion problems for bipartite interconnection networks \textrm{I}: suffcient conditions, Networks, 59(4):349-356, 2012.

\bibitem{7} E.~Cheng, R.~Jia, and D.~Lu, Matching preclusion and conditional matching preclusion for augmented cubes, J. Interconnection Networks, 11:35-60, 2011.

\bibitem{8} E.~Cheng, L.~Lesniak, M.~J.~Lipman, and L.~Lipt\'{a}k, Conditional matching preclusion sets, Inform. Sci., 179:1092-1101, 2009.

\bibitem{9} E.~Cheng, L.~Lesniak, M.~J.~Lipman, and L. Lipt\'{a}k, Matching preclusion for alternating group graphs and their generalizations, Internat. J. Found. Comput. Sci., 19:1413-1437, 2008.

\bibitem{10} E.~Cheng, M.~J.~Lipman, and L.~Lipt\'{a}k, Matching preclusion and conditional matching preclusion for regular interconnection networks, Discrete Appl. Math., 160(13-14):1936-1954, 2012.

\bibitem{12} E.~Cheng, and L.~Lipt\'{a}k, Conditional matching preclusion for Cayley graphs generated by transposition trees, Congr. Numer., 213:143-154, 2012.

\bibitem{14} E.~Cheng, and L.~Lipt\'{a}k, Matching preclusion and conditional matching preclusion problems for tori and related cartesian products, Discrete Appl. Math., 160:1699-1716, 2012.

\bibitem{15} E.~Cheng, and L.~Lipt\'{a}k, Matching preclusion for some interconnection networks, Networks, 50:173-180, 2007.

\bibitem{16} E.~Cheng, L.~Lipt\'{a}k, M.~J.~Lipman, and M.~Toeniskoetter, Conditional matching preclusion for the alternating group graphs and split-stars, Int. J. Comput. Math., 88(6):1120-1136, 2011.

\bibitem{17} E.~Cheng, L.~Lipt\'{a}k, N.~Prince, and K.~Stanton, Matching preclusion and conditional matching preclusion problems for the generalized Petersen graph $P(n,3)$, Congr. Numer., 210:61-72, 2011.

\bibitem{18} E.~Cheng, L.~Lipt\'{a}k, B.~Scholten, and J.~Voss, James, Matching preclusion and conditional matching preclusion for wrapped-around butterfly graphs, Congr. Numer., 204:45-64, 2010.

\bibitem{19} E.~Cheng,L.~Lipt\'{a}k, and D.~Sherman, Matching preclusion for the ($n$,$k$)-bubble-sort graphs, Int. J. Comput. Math., 87(11):2408-2418, 2010.

\bibitem{20} E.~Cheng, S.~Padmanabhan, and K.~Qiu, Matching preclusion and conditional matching preclusion for crossed cubes, Parallel Process. Lett., 22(2):1250005, 2012.

\bibitem{21} X.~Hu, and H.~Liu, The (conditional) matching preclusion for burnt pancake graphs, Discrete Appl. Math., 161(10-11):1481-1489, 2013.

\bibitem{LHZ}Q. Li, J. He, H. Zhang,  Matching preclusion for vertex-transitive networks, Discrete Appl. Math.  207: 90-98, 2016.

\bibitem{22} H.~L\"{u}, X.~Li, and H.~Zhang, Matching preclusion for balanced hypercubes, Theoret. Comput. Sci., 465:10-20, 2012.

\bibitem{24} S.~Wang, R.~Wang, S.~Lin, and J.~Li, Matching preclusion for $k$-ary $n$-cubes, Discrete Appl. Math., 158:2066-2070, 2010.

\end{thebibliography}
\end{document}